\documentclass[a4paper,11pt]{amsart}
\usepackage{color, comment}

\usepackage{tikz} 

\usepackage[foot]{amsaddr} 

\newtheorem{theorem}{Theorem}
\newtheorem{lemma}[theorem]{Lemma}

\newtheorem{proposition}[theorem]{Proposition}
\newtheorem{definition}[theorem]{Definition}

\newtheorem{assumption}[theorem]{Assumption}

\newtheorem{remark}[theorem]{Remark}

\begin{document}

  \title[Multidimensional SDEs with distributional drift]{Multidimensional stochastic differential equations  with distributional drift}
\author{Franco Flandoli$^1$}
\address{$^1$Dipartimento Matematica, Largo Bruno Pontecorvo 5, C.A.P. 56127, Pisa, Italia \texttt{flandoli@dma.unipi.it}}
\author{Elena Issoglio$^2$}
\address{$^2$Department of Mathematics, University of Leeds, Leeds, LS2 9JT, UK \texttt{E.Issoglio@leeds.ac.uk}}
\author{Francesco Russo$^3$}
\address{$^3$Unit\'e de Math\'ematiques appliqu\'ees, ENSTA ParisTech,
Universit\'e Paris-Saclay,
828, boulevard des Mar\'echaux, F-91120 Palaiseau, France \texttt{francesco.russo@ensta-paristech.fr}}
\date{\today}

\begin{abstract}
This paper investigates a time-dependent multidimensional stochastic
differential equation with drift being a distribution
in a suitable class of Sobolev spaces with negative derivation order.
This is done through a careful analysis of the corresponding
Kolmogorov equation whose coefficient is a distribution.
\end{abstract}

\maketitle


\bigskip
{\it AMS-classification}: 60H10; 35K10; 60H30; 35B65.

\section{Introduction}

Let us consider a distribution valued function
$b:[0,T] \rightarrow {\mathcal S}'(\mathbb R^d)$,
where ${\mathcal S}'(\mathbb R^d)$ is the space
of tempered distributions.
An ordinary differential equation
of the type
\begin{equation} \label{EFirstOrder}
\mathrm dX_t = b(t,X_t) \mathrm dt,\quad X_0 = x_0,
\end{equation}
$x_0 \in \mathbb R^d$, $t\in [0,T]$,
does not make sense, except if we consider
it in a very general context of {\it generalized functions}.
Even if $b$ is function valued, without a minimum
regularity in space, problem \eqref{EFirstOrder},
is generally not well-posed.
A motivation for studying \eqref{EFirstOrder}
is for instance to consider $b$ as a quenched  realization
of some (not necessarily Gaussian) random field.
In the annealed form, \eqref{EFirstOrder}  is a singular {\it passive tracer}
 type equation.
 
Let us consider now   equation \eqref{EFirstOrder} with a  noise perturbation,
which is expected to have a regularizing effect, i.e.,
\begin{equation} \label{SDE}
\mathrm dX_t = b(t,X_t) \mathrm dt + \mathrm dW_t,\quad X_0 = x_0,
\end{equation}
for $t\in [0,T]$, where $W$ is a standard $d$-dimensional Brownian motion.
Formally speaking, the Kolmogorov equation associated
with the stochastic differential equation \eqref{SDE} is
\begin{equation} \label{ESecondOrder}
\left \{
\begin{array}{lr}
\partial_t u =b\cdot \nabla u + \frac{1}{2} \Delta u \qquad&  {\rm on} \
[0,T] \times \mathbb R^d,  \\
u(T,\cdot) = f  &{\rm on} \ \mathbb R^d,
\end{array}
\right.
\end{equation}
for suitable final conditions $f$. Equation \eqref{ESecondOrder} was studied
in the one-di\-men\-sio\-nal setting for instance
by \cite{russo_trutnau07} for any time independent $b$ which is the derivative in the distributional sense of
a continuous function
and in the multidimensional setting
by  \cite{i1},
for a class of $b$ of gradient type belonging
to a given  Sobolev space with negative derivation order.
The equation in \cite{i1} involves the \textit{pointwise product} of  distributions 
which in the literature is defined by means of paraproducts.

The point of view of the present paper is to
keep the same interpretation of the product as in \cite{i1} and
 to exploit the solution of a PDE of the same nature
as \eqref{ESecondOrder} in order to give sense and
study solutions of \eqref{SDE}.
A solution $X$ of \eqref{SDE}
is often identified as a {\it diffusion with distributional
drift}.\\
Of course the sense of equation \eqref{SDE} has to be made precise.
 The type of solution we consider will be called {\it virtual
solution}, see Definition  \ref{def: virtual solution}.
That solution will fulfill in particular  the property to be
 the limit in law,
when $n \rightarrow \infty$, of solutions
to classical stochastic differential equations
\begin{equation}  \label{SDEn}
\mathrm dX^n_t = \mathrm dW_t +  b_n(t,X^n_t) \mathrm dt, \quad t\in[0,T],
\end{equation}
where $b_n = b \star \phi_n$ and $(\phi_n)$
is a sequence of mollifiers converging
to the Dirac measure.

Diffusions in the generalized sense were studied by
several authors beginning with, at least in our
knowledge  \cite{portenko}; later on, many
authors considered special cases of stochastic differential equations with generalized coefficients, it is
difficult to quote  them all: in particular, we refer  to the case when $b$ is a
measure, \cite{blei, esd, o, trutnau2}.
\cite{blei} has even considered the case when  $b$ is a not necessarily
locally finite signed measure and the process is a 
possibly exploding semimartingale.
In all these
cases solutions were semimartingales.
In fact, \cite{ew} considered special cases of non-semimartingales
solving stochastic differential equations with generalized drift; those cases include examples
coming from Bessel processes.

The case of time independent SDEs in dimension one
of the type
\begin{equation}  \label{SDEone}
\mathrm dX_t = \sigma(X_t) \mathrm dW_t +  b(X_t)\mathrm dt,\quad t\in[0,T],
\end{equation}
where $\sigma$ is a strictly positive continuous
function and $b$ is the derivative of
a real continuous function
was solved and analyzed carefully in
 \cite{frw2} and \cite{frw1}, which  treated well-posedness
 of the martingale   problem,   It\^o's formula under weak conditions,
 semimartingale
 characterization and Lyons-Zheng
decomposition.
The only supplementary assumption was the
 existence of the function $\Sigma (x) = 2 \int_0^x \frac{b}{\sigma^2} \mathrm dy$
as limit of  appropriate regularizations.
 Also in \cite{basschen1} the authors were interested in
\eqref{SDE} and they provided a well-stated framework when $\sigma$
 and $b$ are $\gamma$-H\"older
 continuous,  $ \gamma > \frac{1}{2}$.
 In \cite{russo_trutnau07} the authors have also shown that in some cases strong solutions
(namely solutions adapted to the completed Brownian filtration) exist and pathwise uniqueness holds.\\
As far as the multidimensional case is concerned,
it seems that the first paper   was  \cite{basschen2}. Here the authors  have focused on \eqref{SDE}
in the case of a time independent drift $b$ which
is a measure of Kato class.

Coming back to the one-dimensional case,
the main idea of \cite{frw1} was the so called Zvonkin transform
 which allows to transform the candidate solution process
$X$ into a solution of a stochastic differential equation
with continuous non-degenerate coefficients without drift.
Recently \cite{karatzas} has considered other types of transforms
to study similar equations.
Indeed the transformation introduced  by Zvonkin in  \cite{z},
when the drift is a function, is also stated in the multidimensional
case. In a series of papers the first named author and coauthors
(see for instance \cite{priola}),
have efficiently made use of a (multidimensional) Zvonkin type transform
for the study of an SDE with measurable
not necessarily bounded drift, which however is still a function.
Zvonkin transform consisted there to transform
 a solution $X$ of \eqref{SDE} (which makes sense
being a classical SDE) through a solution
$\varphi:[0,T] \times \mathbb R^d \rightarrow \mathbb R^d$ of a PDE which
 is close
to the associated Kolmogorov equation \eqref{ESecondOrder} with some suitable
final condition. The resulting process $Y$ with  $Y_t = \varphi(t,X_t)$ for  $t\in [0,T]$
is a solution of an SDE for which one can show strong
 existence and pathwise uniqueness.

Here we have imported that method for the study
of our time-dependent multidimensional SDE
with distributional drift.

The paper is organized as follows. In Section \ref{sc: the komogorov PDE}
we adapt the techniques of \cite{i1}, based on pointwise products
for investigating existence and uniqueness
for a well chosen PDE of the same type as \eqref{ESecondOrder},
see \eqref{eq: backwards PDE}.
In Section \ref{SVirtual} we introduce the notion of {\it virtual solution} of \eqref{SDE}.
The construction will be based on the transformation  $X_t = \psi(t,Y_t)$ for  $t\in [0,T]$, where $Y$ is the solution of \eqref{eq: SDE for Yt} and
$\varphi(t,x) = x + u(t,x),
(t,x) \in [0,T] \times \mathbb R^d$, with $u$
being the solution of \eqref{eq: backwards PDE}.
Section \ref{SFinal}
shows that the virtual solution is indeed the limit
of classical solutions of regularized stochastic
differential equations.

\section{The Kolmogorov PDE}\label{sc: the komogorov PDE}

\subsection{Setting and preliminaries}
 Let $b$ be a vector field on $[0,T]\times \mathbb R^d, d\geq 1$, which is a distribution in space and weakly bounded in time, that is $b\in L^\infty ([0,T];\mathcal S'( \mathbb R^d ; \mathbb R^d))$. Let $\lambda>0$. We consider the parabolic PDE in $[0,T]\times \mathbb R^d$
\begin{align}\label{eq: backwards PDE}
\left\{
\begin{array}{lr}
 \partial_t u+ L^b u -(\lambda+1) u = -b \quad & \text{on } [0,T]\times\mathbb R^d, \\
 u(T)=0 & \text{on } \mathbb R^d,
\end{array}
\right.
\end{align}
where $L^b u = \frac12 \Delta u  +b \cdot \nabla u$ has to be interpreted componentwise, that is $(L^b u)_i = \frac12 \Delta u_i  +b \cdot \nabla u_i$ for $i=1,\ldots,d$. A continuous function $u:[0,T]\times \mathbb R^d\to \mathbb R^d$ will also be considered without any comment as $u:[0,T]\to C(\mathbb R^d; \mathbb R^d)$. In particular we will write $u(t,x)=u(t)(x)$ for all $(t,x)\in [0,T]\times \mathbb R^d$.
\begin{remark}
All the results we are going to prove remain valid for the equation%
\begin{align*}
\left\{
\begin{array}{lr}
\partial_{t}u+L^{b_{1}}u-\left(  \lambda+1\right)  u   =-b_{2}\quad  & \text{on
}\left[  0,T\right]  \times\mathbb{R}^{d},\\
u\left(  T\right)     =0 & \text{on } \mathbb R^d,
\end{array}
\right.
\end{align*}
 where $b_{1},b_{2}$ both satisfy the same assumptions as $b$. We restrict the discussion to the case $b_{1}=b_{2}=b$ to avoid notational confusion in the subsequent sections.
\end{remark}
Clearly we have to specify the meaning of the \emph{product} $b\cdot \nabla u_i$ as $b$ is a distribution. In particular, we are going to  make use in an essential way the notion of paraproduct, see \cite{r1}. We recall below a few elements of this theory;\ in particular, when we say that the pointwise product exists in $\mathcal{S}^{\prime}$ we mean that the limit
 \eqref{eq: paraproduct} exists in $\mathcal{S}^{\prime}$.
 For shortness we denote by $\mathcal S'$ and $\mathcal S$ the spaces $\mathcal{S}^{^{\prime}}( \mathbb R^d ; \mathbb R^d) $ and $\mathcal{S}( \mathbb R^d ; \mathbb R^d) $ respectively. Similarly for the $L^p$-spaces, $1\leq p\leq \infty$. We denote by $\langle\cdot, \cdot\rangle$ the dual pairing between an element of $\mathcal S'$ and an element of $\mathcal S$.

We now recall a definition of a pointwise product between a function and a distribution (see e.~g.~\cite{r1}) and some useful properties.\\
Suppose we are given $f\in \mathcal{S} ' (\mathbb{R}^d) $. Choose a function $\psi\in \mathcal{S}(\mathbb{R}^d)$ such that $0\leq \psi(x)\leq 1$ for every $x\in \mathbb{R}^d$, $\psi(x)=1$ if $\vert x\vert\leq 1$ and $\psi(x)=0$ if $\vert x\vert\geq \frac{3}{2}$. Then consider the following approximation $S^j f$ of $f$ for each $j\in \mathbb N$
\[
S^j f (x):= \left(\psi\left(\frac{\xi}{2^j}\right)\hat f\right)^\vee (x),
\]
that is in fact the convolution of $f$ against the smoothing rescaled
 function  $\psi_j$ associated with $\psi$.
This approximation is used  to define the product $fg$ of two distributions $f,g\in\mathcal S'$ as follows:
\begin{equation}\label{eq: paraproduct}
fg:=\lim _{j\to \infty} S^jfS^jg,
\end{equation}
if the limit exists in $\mathcal{S}'(\mathbb{R}^d)$. The convergence in the case we are interested in is part of the assertion below (see \cite{h1} appendix C.4, \cite{r1} Theorem 4.4.3/1).

\begin{definition}\label{def: mild solution in S'}
Let $b,u:\left[  0,T\right]  \rightarrow\mathcal{S}^{^{\prime}}$ be such that
\begin{itemize}
\item[(i)]the pointwise product $b\left(  t\right)  \cdot\nabla u\left(  t\right)  $
exists in $\mathcal{S}^{^{\prime}}$ for a.e. $t\in\left[  0,T\right],
$
\item[(ii)] there are $r\in\mathbb{R}$, $q\geq1$ such that $b,u,b\cdot\nabla u\in
L^{1}\left(  [0,T];H_{q}^{r}  \right)  $.
\end{itemize}
We say that $u$ is a mild solution of equation \eqref{eq: backwards PDE} in $\mathcal{S}^{^{\prime}}$ if, for every $\psi\in\mathcal{S}$ and $t\in\left[  0,T\right]  $, we have%
\begin{align} \label{weak-mild}%
\left\langle u\left(  t\right)  ,\psi\right\rangle & = \int_{t}^{T}\left\langle
b\left(  r\right)  \cdot\nabla u\left(  r\right)  ,P\left(  r-t\right)
\psi\right\rangle\mathrm dr\\ \nonumber
&+\int_{t}^{T}\left\langle b\left(  r\right)  -\lambda
u\left(  r\right)  ,P\left(  r-t\right)  \psi\right\rangle\mathrm dr.
\end{align}
\end{definition}

Here $\left(  P\left(  t\right)  \right)  _{t\geq0}$ denotes the heat
semigroup on $\mathcal{S}$ generated by $\frac{1}{2}\Delta-I$, defined for
each $\psi\in\mathcal{S}$ as%
\[
\left(  P\left(  t\right)  \psi\right)  \left(  x\right)  =\int_{\mathbb{R}%
^{d}}p_{t}\left(  x-y\right)  \psi\left(  y\right) \mathrm dy,
\]
where $p_t(x)$ is the heat kernel $p_t(x)= e^{-t} \frac{1}{(2t\pi )^{d/2}} \exp{\left(- \frac{|x|_d^2}{2t}\right)}$ and $\vert \cdot \vert_d$ 
is the usual Euclidean norm in $\mathbb R^d$.
The semigroup $\left( P(  t)\right)_{t\geq 0}  $ extends to $\mathcal{S}^{\prime}$, where
it is defined as
\[
\left(  P_{\mathcal{S}^{\prime}}\left(  t\right)  h\right)  \left(
\psi\right)  =  \langle h, \int_{\mathbb R^d} p_t(\cdot-y) \psi(y) \mathrm d y \rangle,
\]
for every $h\in\mathcal{S}^{\prime}$, $\psi\in\mathcal{S}$.

The fractional Sobolev spaces $H^r_q$ are the so called Bessel potential spaces and will be defined in the sequel.
\begin{remark}
If $b,u,b\cdot\nabla u$ a priori belong to spaces $L^{1}\left(  [0,T];H_{q_{i} }^{r_{i}} \right)  $ for different $r_{i}%
\in\mathbb{R}$, $q_{i}\geq1$, $i=1,2,3$, then (see e.g. \eqref{eq: embedding property}) there exist common
$r\in\mathbb{R}$, $q\geq1$ such that $b,u,b\cdot\nabla u\in L^{1}\left([0,T];H_{q}^{r} \right)$.
\end{remark}

The semigroup $\left(  P_{\mathcal{S}^{\prime}}\left(  t\right)  \right)
_{t\geq0}$ maps any $L^{p}\left(  \mathbb{R}^{d}\right)  $ into itself, for
any given $p\in\left(  1,\infty\right)  $;\ the restriction $\left(
P_{p}\left(  t\right)  \right)  _{t\geq0}$ to $L^{p}\left(  \mathbb{R}%
^{d}\right)  $ is a bounded analytic semigroup, with generator $-A_{p}$, where
$A_{p}=I-\frac{1}{2}\Delta$, see \cite[Thm.~1.4.1, 1.4.2]{d1}. The fractional powers of $A_{p}$ of order $s\in\mathbb{R}$ are then well-defined, see \cite{pazy83}. The fractional
Sobolev spaces $H^s_p(\mathbb R^d)$ of order $s\in \mathbb R$ are then $H^s_p(\mathbb R^d): =A_p^{- s/2}(L^p(\mathbb R^d))$ for all $s\in \mathbb R$ and they are Banach spaces when endowed with the norm $\|\cdot\|_{H^s_p}= \|A_p^{s/2} (\cdot)\|_{L^p}$.
The domain of $A_p^{s/2}$ is then the Sobolev space of order $s$, that is $D (A_p^{s/2})= H^s_p(\mathbb R^d)$, for all $s\in \mathbb R$. Furthermore, the negative powers $A_p^{-s/2}$ act as isomorphism from $H^{\gamma}_p(\mathbb R^d)$ onto $H^{\gamma+s}_p(\mathbb R^d)$ for $\gamma\in \mathbb R$.

We have defined so far function spaces and operators in the case of scalar
valued functions. The extension to vector valued functions must be understood
componentwise. For instance, the space $H_{p}^{s}\left(  \mathbb{R}%
^{d},\mathbb{R}^{d}\right)  $ is the set of all vector fields $u:\mathbb{R}%
^{d}\rightarrow\mathbb{R}^{d}$ such that $u^{i}\in H_{p}^{s}\left(
\mathbb{R}^{d}\right)  $ for each component $u^{i}$ of $u$; the vector field
$P_{p}\left(  t\right)  u:\mathbb{R}^{d}\rightarrow\mathbb{R}^{d}$ has
components $P_{p}\left(  t\right)  u^{i}$, and so on. Since we use vector
fields more often than scalar functions, we shorten some of the notations: we
shall write $H_{p}^{s}$ for $H_{p}^{s}\left(  \mathbb{R}^{d},\mathbb{R}%
^{d}\right)  $. Finally, we denote by $H_{p,q}^{-\beta}$ the space $H_{p}^{-\beta}\cap H_{q}^{-\beta}$ with the usual norm.

\begin{lemma}\label{lm: pointwise product}
 Let $1<p,q<\infty$ and $0<\beta<\delta$ and assume that $q>p \vee\frac{d}{\delta}$. Then for every $f\in H^{\delta}_p(\mathbb{R}^d)$ and $g\in H^{-\beta}_q(\mathbb{R}^d)$ we have  $fg\in H_p^{-\beta}(\mathbb{R}^d) $ and there exists a positive constant $c$ such that
\begin{equation}\label{eq: pointwise product}
 \| fg \|_{ H_p^{-\beta}(\mathbb{R}^d)}\leq c \| f\|_{H_p^{\delta}(\mathbb{R}^d)}\cdot \| g\|_{ H_q^{-\beta}(\mathbb{R}^d)}.
\end{equation}
\end{lemma}

For the following the reader can also consult  \cite[Section 2.7.1]{triebel78}. Let us consider the spaces $C^{0,0}(\mathbb R^d; \mathbb R^d)$ and $C^{1,0}(\mathbb R^d; \mathbb R^d)$ defined as the closure of  $\mathcal S$ with respect to the norm  $\|f\|_{C^{0,0}} = \|f\|_{L^\infty}$ and $\|f\|_{C^{1,0}} = \|f\|_{L^\infty} + \|\nabla f\|_{L^\infty}$, respectively.
For  $0 < \alpha < 1$ we will consider the Banach spaces
\begin{align*}
&C^{0, \alpha}= \{ f \in C^{0,0}(\mathbb R^d; \mathbb R^d) :  \|f\|_{C^{0,\alpha}} <\infty \},\\
&C^{1, \alpha} =\{  f \in C^{1,0}(\mathbb R^d; \mathbb R^d) :  \|f\|_{C^{1,\alpha}} <\infty \},
\end{align*}
endowed with the norms
\begin{align*}
&\|f\|_{C^{0,\alpha}} := \| f \|_{L^\infty}  + \sup_{x\neq y \in \mathbb R^d}  \frac{|f(x)-f(y)|}{|x-y|^\alpha} \\
& \|f\|_{C^{1,\alpha}} := \| f \|_{L^\infty}  + \| \nabla f \|_{L^\infty} + \sup_{x\neq y \in \mathbb R^d}  \frac{|\nabla f(x)-\nabla f(y)|}{|x-y|^\alpha},
\end{align*}
respectively.

From    now on,  we are going to make the following \textbf{standing assumption} on the drift $b$ and on the possible choice of parameters:
\begin{assumption}\label{assumption}
 Let $\beta\in\left(  0,\frac{1}{2}\right)  $, $q\in\left(  \frac{d}{1-\beta},\frac{d}{\beta}\right)  $ and set $\tilde q:= \frac{ d}{1-\beta}$. The drift $b$ will always be of the type
\[
b\in L^\infty\left ([0,T]; H^{-\beta}_{\tilde q,q} \right).
\]
\end{assumption}
\begin{remark}\label{rem: b in L_p}
The fact that $b\in L^\infty\left ([0,T]; H^{-\beta}_{\tilde q, q}\right)$ implies,
 for each  $p\in [\tilde q, q]$, that $b\in L^\infty\left ([0,T]; H^{-\beta}_{p} 
\right)$.
\end{remark}

\begin{figure}
\begin{tikzpicture}
\draw[->] (-1,0)--(7,0) node[below right] {$\frac1p$};
\draw [->] (0.5,-1)--(0.5,3) node[left] {$\delta$};
\fill[gray] (2, .9)--(2,2.3)--(4.33,2.3)--(2, .9);
\draw (-.5,.3)  node[anchor=east]{$\beta$}--(5.5,.3);
\draw (-.5,2.3) node[anchor=east]{$1-\beta$} --(5.5,2.3);
\draw (2,-.7) node[anchor=south west]{$ \frac1q$} --(2,2.8);
\draw[dashed] (1,-.7) node[anchor=south west]{$ \frac{ \beta}{ d}$} --(1,2.8);
\draw[dashed] (4.33,-.7) node[anchor=south west]{$\frac{1-\beta}{ d}=:\frac{1}{\tilde q}$} --(4.33,2.8);
\draw (.5,0)--(5,2.7);  
\end{tikzpicture}
\caption{The set $K(\beta, q)$.}\label{fig: the set K}
\end{figure}

Moreover we consider the  set
\begin{equation}\label{eq: set admissible kappa}
K(\beta,q):=\left\{\kappa=(\delta, p): \; \beta< \delta< 1-\beta, \,  \frac d\delta < p < q  \right\},
\end{equation}
which is drawn in  Figure \ref{fig: the set K}. Note that $K(\beta,q)$ is nonempty since $\beta<\frac12 $ and $ \frac d{1-\beta}< q < \frac d\beta $.

\begin{remark}\label{rem: pointwise product}
As a consequence of Lemma \ref{lm: pointwise product}, for $0<\beta<\delta$ and  $q>p\vee \frac{d}{\delta}$  and if
  $ b\in L^\infty([0,T]; H^{-\beta}_q)$ and  $u \in C^{0}([0,T]; H^{1+\delta}_p)$,  then for all $t\in [0,T]$ we have $b(t)\cdot\nabla u(t) \in  H^{-\beta}_p $ and
\[
\|b(t)\cdot\nabla u(t) \|_{ H^{-\beta}_p} \leq c \|b\|_{\infty, H^{-\beta}_q } \|u(t) \|_{ H^{\delta}_p },
\]
having used the continuity of $\nabla$ from $ H^{1+\delta}_p $ to $ H^{\delta}_p$. Moreover any choice $(\delta,p)\in K(\beta, q)$ satisfies the hypothesis in Lemma  \ref{lm: pointwise product}.
\end{remark}

\begin{definition}\label{def: mild sol in H}
Let $(\delta, p)\in K(\beta,q) $. We say that $u\in C\left(  \left[  0,T\right]
;H_{p}^{1+\delta}  \right)  $ is a mild solution of equation \eqref{eq: backwards PDE} in $H_{p}^{1+\delta}$ if%
\begin{equation}\label{mild in Sobolev}
u\left(  t\right)  =\int_{t}^{T}P_{p}\left(  r-t\right)  b\left(  r\right)
\cdot\nabla u\left(  r\right) \mathrm dr+\int_{t}^{T}P_{p}\left(  r-t\right)  \left(
b\left(  r\right)  -\lambda u\left(  r\right)  \right)\mathrm dr,
\end{equation}
for every $t\in\left[  0,T\right]  $.
\end{definition}

\begin{remark}\label{remark 2}
Notice that $b\cdot\nabla u\in L^{\infty}\left(  \left[ 0,T\right]  ;H_{p}^{-\beta}\right)$   by Remark \ref{rem: pointwise product}. By Remark \ref{rem: b in L_p}, $b\in L^{\infty}\left(  \left[  0,T\right]  ;H_{p}^{-\beta} \right) $. Moreover $\lambda u\in L^{\infty}\left(  \left[  0,T\right]  ;H_{p}^{-\beta} \right)  $ by the embedding $H_p^{1+\delta}  \subset H_p^{-\beta}$. Therefore the integrals in Definition \ref{def: mild sol in H} are meaningful in $H_{p}^{-\beta}$.
\end{remark}

Note that setting $v(t,x):= u(T-t,x)$, the PDE \eqref{eq: backwards PDE} can be equivalently rewritten  as
\begin{align}\label{eq: fwd PDE}
\left\{
\begin{array}{lr}
 \partial_tv =L^b v -(\lambda+1) v+ b \quad & \text{on } [0,T]\times\mathbb R^d, \\
 v(0)=0 & \text{on } \mathbb R^d \, .
\end{array}
\right.
\end{align}
The notion of mild solutions of equation \eqref{eq: fwd PDE} in $\mathcal S'$ and in $ H_{p}^{1+\delta} $ are analogous to Definition \ref{def: mild solution in S'} and  Definition \ref{def: mild sol in H}, respectively. In particular the mild solution in $ H_{p}^{1+\delta}$ verifies
\begin{equation}\label{eq: mild solution to fwd PDE}
 v(t)=\int_0^t P_p(t-r)  \left(b(r)\cdot\nabla v(r) \right)\mathrm dr  +\int_0^t P_p(t-r) (b(r)-\lambda v(r) ) \mathrm dr.
\end{equation}
Clearly the regularity properties of $u$ and $v$ are the same.\\
For a Banach space $X$ we denote the usual norm in $L^\infty([0,T]; X) $ by $\|f\|_{\infty, X}$ for $f\in L^\infty([0,T]; X)$. Moreover, on the Banach space $ C ([0,T]; X)$ with norm $\|f\|_{\infty, X}:=\sup_{0\leq t\leq T}\|f(t)\|_X$ for $f\in  C ([0,T];X)$,  we introduce a family of equivalent norms $\{\left\|\cdot\right\|^{(\rho)}_{\infty, X} ,\ \rho \geq 1\} $ as follows:
\begin{equation}\label{eq: equivalent norm}
\|f\|^{(\rho)}_{\infty, X}:= \sup_{0\leq t\leq T}\mathrm e^{-\rho t} \|f(t)\|_X.
\end{equation}

Next we state a mapping property of the heat semigroup $P_p(t)$ on $L^p(\mathbb R^d)$: it maps distributions of fractional order $-\beta$ into functions of fractional order $1+\delta$ and the price one has to pay is a singularity in time. The proof is analogous to the one in \cite[Prop.~3.2]{i1} and is based on the analyticity of the semigroup.
\begin{lemma}\label{lm: semigroup from -beta to 1+delta}
Let $0<\beta<\delta$, $\delta+\beta<1$ and $w\in H^{-\beta}_p(\mathbb R^d)$. Then $P_p(t) w\in  H_p^{1+\delta}(\mathbb R^d)$ for any $t>0$ and moreover there exists a positive constant $c$ such that
\begin{equation} \label{eq: semigroup from -beta to 1+delta}
\left\|P_p(t) w\right\|_{H_p^{1+\delta}(\mathbb R^d)}\leq c\left\|w\right\|_{H_p^{-\beta}(\mathbb R^d)} t^{-\frac{1+\delta+\beta}{2}}.
\end{equation}
\end{lemma}

\begin{proposition}\label{pr: holder continuity with b bounded}
Let $f\in L^{\infty}\left( [ 0,T]; H^{-\beta}_p \right)  $ and $g:\left[  0,T\right]
\rightarrow H^{-\beta}_p$ for $\beta\in\mathbb R$ defined as%
\[
g\left(  t\right)  =\int_{0}^{t} P_p(t-s) f\left(  s\right) \mathrm d s.
\]
Then $g\in C^{\gamma}\left(  \left[  0,T\right]  ;   H^{2-2\epsilon-\beta }_p \right)  $ for every $\epsilon>0$ and $\gamma\in\left(
0,\epsilon\right)  $.
\end{proposition}

\begin{proof}
First observe that for $f\in D(A_p^\gamma)$  there exists $C_\gamma>0$ such that
\begin{equation}\label{eq: interp}
\|P_p(t) f -f\|_{L^p} \leq C_\gamma t^\gamma \|f\|_{H^{2\gamma}_p},
\end{equation}
for all $t\in\left[0,T\right]  $ (see \cite[Thm 6.13, (d)]{pazy83}).

Let $0\leq r<t\leq T$. We have
\begin{align*}
 g(t)-g(r) =& \int_{0}^{t} P_p(t-s) f\left(  s\right) \mathrm d s - \int_{0}^{r} P_p(r-s) f\left(  s\right) \mathrm d s \\
 =& \int_{r}^{t} P_p(t-s) f\left(  s\right) \mathrm d s +\int_{0}^{r} \left(P_p(t-s)-P_p(r-s)\right) f\left(  s\right) \mathrm d s\\
 =& \int_{r}^{t} P_p(t-s) f\left(  s\right) \mathrm d s \\
 &+\int_{0}^{r} A_p^\gamma P_p(r-s) \left(A_p^{-\gamma} P_p(t-r) f(s) - A_p^{-\gamma}  f\left(  s\right) \right) \mathrm d s,
\end{align*}
so that
\begin{align*}
\| &g(t)- g(r) \|_{H_p^{2-2\epsilon -\beta}} \\
\leq& \int_{r}^{t}  \| P_p(t-s) f( s) \|_{H^{2-2\epsilon-\beta }_p} \mathrm d s \\
&+\int_{0}^{r} \| A_p^\gamma P_p(r-s) \left(A_p^{-\gamma} P_p(t-r) f(s) - A_p^{-\gamma}  f\left(  s\right) \right)\|_{H^{2-2\epsilon-\beta }_p} \mathrm d s\\
\leq& \int_{r}^{t}  \|A_p^{1-\epsilon-\beta/2} P_p(t-s) f( s) \|_{L^p} \mathrm d s \\
&+\int_{0}^{r} \| A_p^{1-\epsilon-\beta/2+\gamma} P_p(r-s) \left(A_p^{-\gamma} P_p(t-r) f(s) - A_p^{-\gamma} f(s) \right)\|_{L^p} \mathrm d s\\
=&:\text{(S1)} + \text{(S2)}.
\end{align*}
Let us consider (S1) first. We have
\begin{align*}
 \text{(S1)}\leq & \int_{r}^{t}  \|A_p^{1-\epsilon} P_p(t-s)\|_{L^p\to L^p} \| A^{-\beta/2} f( s) \|_{L^p} \mathrm d s \\
 \leq & \int_{r}^{t} C_\epsilon   (t-s)^{-1+\epsilon}  \| f( s) \|_{ H^{-\beta}_p} \mathrm d s \\
 \leq & C_\epsilon   (t-r)^{\epsilon}  \|f\|_{\infty, H^{-\beta}_p},
\end{align*}
having used \cite[Thm 6.13, (c)]{pazy83}. Using again the same result, the term (S2), together with \eqref{eq: interp}, gives (with the constant $C$ changing from line to line)
\begin{align*}
 \text{(S2)}&=\int_{0}^{r} \left\| A_p^{1-\epsilon+\gamma} P_p(r-s) \left( P_p(t-r) A_p^{-\gamma-\beta/2} f(s) - A_p^{-\gamma-\beta/2} f(s) \right)\right\|_{L^p} \mathrm d s\\
 &\leq C \int_{0}^{r}  (r-s)^{-1+\epsilon-\gamma} \left\| P_p(t-r) A_p^{-\gamma-\beta/2} f(s) - A_p^{-\gamma-\beta/2} f(s) \right\|_{L^p} \mathrm d s\\
 &\leq C \int_{0}^{r}  (r-s)^{-1+\epsilon-\gamma}  (t-r)^{\gamma}  \| A_p^{-\gamma -\beta/2} f(s) \|_{H^{2\gamma}_p} \mathrm d s\\
 &\leq C  (t-r)^{\gamma}  \int_{0}^{r}  (r-s)^{-1+\epsilon-\gamma}  \| f(s) \|_{H^{-\beta}_p} \mathrm d s\\
 &\leq C  (t-r)^{\gamma}  \int_{0}^{r}  (r-s)^{-1+\epsilon-\gamma}  \| f\|_{\infty, H^{-\beta}_p} \mathrm d s\\
 &\leq C  (t-r)^{\gamma}   r^{\epsilon-\gamma}  \| f\|_{\infty, H^{-\beta}_p}.
\end{align*}
Therefore we have $g\in C^{\gamma}\left(  \left[  0,T\right]  ;   H^{2-2\epsilon-\beta }_p \right)  $ for each $0<\gamma < \epsilon$ and the proof is complete.
\end{proof}

The following lemma gives integral bounds which will be used later. The proof makes use of the Gamma and the Beta functions together with some basic integral estimates. We recall the definition of the Gamma function:
\[
\Gamma(a)=\int_0^\infty e^{-t} t^{a-1}\mathrm d t,
\]
and the integral converges for any $a\in \mathbb C$ such that $\operatorname{Re}(a)>0$.
\begin{lemma}\label{lm: stima integrale con fz gamma}
If $0\leq s<t\leq T<\infty$ and $0\leq \theta<1$ then for any $\rho\geq 1$ it holds
\begin{equation}\label{eq 1 stima integrale con fz gamma}
 \int_s^t e^{-\rho r} r^{-\theta} \mathrm{d} r\leq\Gamma(1-\theta) \rho^{\theta-1}.
\end{equation}
Moreover if $\gamma>0$ is such that $ \theta+\gamma<1$ then for any $ \rho \geq 1$ there exists a positive constant $C$ such that
\begin{equation}\label{eq 2 stima integrale con fz gamma}
\int_0^{t} e^{-\rho(t- r)} (t-r)^{-\theta} r^{-\gamma} \mathrm{d} r\leq C \rho^{\theta -1+\gamma}.
\end{equation}
\end{lemma}

\begin{lemma}\label{lm: contraction property of integral operator}
 Let $1<p,q<\infty$ and $0<\beta<\delta$ with $q>p \vee\frac{d}{\delta}$ and let $\beta+\delta<1$. Then for $b\in L^\infty([0,T]; H^{-\beta}_{p,q})$ and $v\in C ([0,T]; H^{1+\delta}_p) $ we have
\begin{itemize}
 \item[(i)] $\int_0^\cdot P_{p}(\cdot-r) b(r) \mathrm dr\in C([0,T]; H^{1+\delta}_p)$;
 \item[(ii)]$\int_0^\cdot P_{p}(\cdot-r) \left( b(r)\cdot\nabla v(r) \right) \mathrm dr \in C([0,T]; H^{1+\delta}_p)$  with
\[
\left \| \int_0^\cdot P_{p}(\cdot-r) \left( b(r)\cdot\nabla v(r)\right) \mathrm dr\right \|^{(\rho)}_{\infty, H^{1+\delta}_p} \leq c(\rho) \|v\|^{(\rho)}_{\infty, H^{1+\delta}_p};
\]
 \item[(iii)] $\lambda \int_0^\cdot P_{p}(\cdot-r)  v(r) \mathrm dr   \in C([0,T]; H^{1+\delta}_p)$ with
\[
\left \| \lambda \int_0^\cdot P_{p}(\cdot-r)   v(r) \mathrm dr \right \|^{(\rho)}_{\infty, H^{1+\delta}_p} \leq c(\rho) \|v\|^{(\rho)}_{\infty, H^{1+\delta}_p},
\]
\end{itemize}
where the constant $c(\rho)$ is independent of $v$ and tends to zero as $\rho$ tends to infinity.
\end{lemma}
Observe that $(\delta,p)\in K(\beta, q)$ satisfies the hypothesis in Lemma \ref{lm: contraction property of integral operator}.
\begin{proof}
(i) Lemma  \ref{lm: semigroup from -beta to 1+delta}
 implies  that $P_p(t) b(t)\in H^{1+\delta}_p $ for each $t\in[0,T]$.
Choosing $\epsilon = \frac{1-\beta-\delta}{2}$,
 Proposition \ref{pr: holder continuity with b bounded}
implies item (i).
 
(ii) Similarly to part (i), the first part follows by 
Proposition \ref{pr: holder continuity with b bounded}.
Moreover
\begin{align*}
& \sup_{0\leq t\leq T} \mathrm e^{-\rho t} \left\|\int_0^t P_{p}(t-r)\left( b(r)\cdot\nabla v(r) \right)  \mathrm dr\right\|_{H^{1+\delta}_p} \\
&\leq c \sup_{0\leq t\leq T}\int_0^t  \mathrm e^{-\rho t} (t-r)^{-\frac{1+\delta+\beta}{2}} \|v(r)\|_{H^{1+\delta}_p} \|b(r) \|_{H^{-\beta}_q}\mathrm dr \\
&\leq 	c \|b \|_{\infty,H^{-\beta}_q } \sup_{0\leq t\leq T}\int_0^t   \mathrm e^{-\rho r} \|v(r)\|_{H^{1+\delta}_p}  \mathrm e^{-\rho (t-r)} (t-r)^{-\frac{1+\delta+\beta}{2}}  \mathrm dr \\
&\leq c \|v\|^{(\rho)}_{\infty, H^{1+\delta}_p} \|b \|_{\infty, H^{-\beta}_q} \rho^\frac{\delta+\beta-1}{2} <\infty.
\end{align*}
Thus $\left \|\int_0^\cdot  P_{p}(\cdot-r) \left( b(r)\cdot\nabla v(r) \right)  \mathrm dr \right  \|^{(\rho)}_{\infty, H^{1+\delta}_p}\leq c(\rho) \|v\|^{(\rho)}_{\infty, H^{1+\delta}_p} $.\\
(iii) Similarly to parts (i) and (ii) the continuity property follows
 by Proposition \ref{pr: holder continuity with b bounded}. Then  
\begin{align*}
 \sup_{0\leq t\leq T} \mathrm e^{-\rho t} \left\|\int_0^t P_{p}(t-r) v(r)  \mathrm dr\right\|_{H^{1+\delta}_p}
&\leq c \sup_{0\leq t\leq T}\int_0^t  \mathrm e^{-\rho t} \|v(r)\|_{H^{1+\delta}_p} \mathrm dr \\
&\leq c \|v\|^{(\rho)}_{\infty, H^{1+\delta}_p}  \rho^{-1} <\infty. \qedhere
\end{align*}

\end{proof}

\subsection{Existence}\label{sc: existence}

Let us now introduce the \emph{integral operator} $I_t(v)$ as the right hand side of \eqref{eq: mild solution to fwd PDE}, that is, given any $v\in C ([0,T]; H^{1+\delta}_p)$, we define for all $t\in[0,T]$
\begin{equation}\label{eq: integral operator}
I_t(v):= \int_0^t P_{p}(t-r) \left( b(r)\cdot\nabla v(r) \right) \mathrm dr +\int_0^t P_{p}(t-r)( b(r) - \lambda  v(r)) \mathrm dr.
\end{equation}
By Lemma \ref{lm: contraction property of integral operator}, the integral operator is  well-defined and it is a linear operator on  $ C([0,T]; H^{1+\delta}_p)$.

Let us remark that Definition \ref{def: mild sol in H} is in fact meaningful under the assumptions of Lemma \ref{lm: contraction property of integral operator}, which are more general than the ones of Definition \ref{def: mild sol in H} (see Remark \ref{rem: discussion on parameters}).

\begin{theorem}\label{thm: contraction property of integral operator}
Let $1<p,q<\infty$ and $0<\beta<\delta$ with $q>p \vee\frac{d}{\delta}$ and let $\beta+\delta<1$. Then for $b\in L^\infty([0,T]; H^{-\beta}_{p,q})$ 
there exists a unique mild solution $v$ to the PDE \eqref{eq: mild solution to fwd PDE} in $H^{1+\delta}_p$. Moreover for any $0<\gamma<1-\delta-\beta$ the solution $v$ is in $C^\gamma ([0,T]; H^{1+\delta}_p)$.
\end{theorem}
\begin{proof}
By Lemma \ref{lm: contraction property of integral operator} the integral operator is a contraction for some $\rho$ large enough, thus by the Banach fixed point theorem there exists a unique mild solution $v\in C ([0,T]; H^{1+\delta}_p)$ to the PDE \eqref{eq: mild solution to fwd PDE}. For this solution we obtain H\"older continuity in time of order $\gamma$ for each $0<\gamma< 1-\delta-\beta$. In fact each term on the right-hand side of \eqref{eq: integral operator} is $\gamma$-H\"older continuous by Proposition \ref{pr: holder continuity with b bounded} as $b,b\cdot \nabla v, v \in L^\infty ([0,T]; H^{-\beta}_p)$.
\end{proof}

\begin{remark}\label{rem: discussion on parameters}
By Theorem \ref{thm: contraction property of integral operator} and by the definition of $K(\beta, q)$, for each $(\delta, p)\in K(\beta,q)$ there exists a unique mild solution in  $H^{1+\delta}_p$. However notice that the assumptions of Theorem \ref{thm: contraction property of integral operator} are slightly more general than those of Assumption \ref{assumption} and of the set $K(\beta,q)$. Indeed, the following conditions are not required for the existence of the solution to the PDE (Lemma \ref{lm: contraction property of integral operator} and Theorem \ref{thm: contraction property of integral operator}):
 \begin{itemize}
 \item the condition $ \frac {d}{\delta}<p$ appearing in the definition of the region $K(\beta,q)$ is only needed in order to embed the fractional Sobolev space $H^{1+\delta}_p$ into $C^{1,\alpha}$ (Theorem \ref{thm: fractional Morrey ineq}).
  \item the condition $ q<\frac{d}{\beta} $ appearing in Assumption \ref{assumption} is only needed in Theorem \ref{thm: uniqueness of K solution} in order to show uniqueness for the solution $u$, independently of the choice of  $(\delta,p)\in K(\beta,q)$.
 \end{itemize}
\end{remark}

The following embedding theorem describes how to compare fractional Sobolev spaces with different orders and provides a generalisation of the Morrey inequality to fractional Sobolev spaces. For the proof we refer to \cite[Thm.~2.8.1, Remark 2]{triebel78}.

\begin{theorem}\label{thm: fractional Morrey ineq}
\emph{Fractional Morrey inequality.} Let $0<\delta< 1 $ and $d/\delta<p<\infty$. If $f\in H^{1+\delta}_p(\mathbb R^d)$ then there exists a unique version of $f$ (which we denote again by  $f$) such that $f$ is differentiable. Moreover $ f\in C^{1,\alpha}(\mathbb R^d)$ with $\alpha= \delta-d/p$ and
\begin{equation}\label{eq: fractional Morrey ineq}
\|f\|_{C^{1,\alpha}}\leq c \|f\|_{H^{1+\delta}_p}, \quad  \|\nabla f\|_{C^{0,\alpha}}\leq c \|\nabla f\|_{H^{\delta}_p},
\end{equation}
where $c=c(\delta, p, d)$ is a universal constant.\\
\emph{Embedding property.} For $1<p\leq q <\infty$ and $s-\frac dp \geq t- \frac dq$ we have
\begin{equation}\label{eq: embedding property}
H^s_p(\mathbb R^d) \subset H^t_q(\mathbb R^d).
\end{equation}
\end{theorem}

\begin{remark}
According to the fractional Morrey inequality, if $u(t)\in H^{1+\delta}_p$ then $\nabla u(t)\in C^{0,\alpha}$ for $\alpha = \delta- d/p$ if $p>d/\delta$. In this case the condition on the pointwise product $q>p\vee d/\delta$ reduces to $q>p$.
\end{remark}

\subsection{Uniqueness}
 In this section we show that the solution $u$ is unique, independently of the choice of $(\delta,p)\in K(\beta,q)$.

\begin{lemma}\label{lm: mild solutions coincide in S'}
Let $u$ be a mild solution in $\mathcal{S}^{^{\prime}}$ such that $u\in C\left(  \left[  0,T\right]  ;H_{p}^{1+\delta}  \right)  $ for some $(\delta, p) \in K(\beta, q)$. Then $u$ is a mild solution of  \eqref{eq: backwards PDE} 
in $H_{p}^{1+\delta}$.
\end{lemma}

\begin{proof}
As explained in Remark \ref{remark 2},  $b\cdot\nabla u,b,\lambda
u\in L^{\infty}\left(  \left[  0,T\right]  ;H_{p}^{-\beta}  \right)  $. Given $\psi\in\mathcal{S}$ and $h\in
H_{p}^{-\beta}  $, we have%
\begin{equation}
\left\langle h,P\left(  s\right)  \psi\right\rangle =\left\langle
P_{p}\left(  s\right)  h,\psi\right\rangle, \label{duality}%
\end{equation}
for all $s\geq0$. Indeed, $P_{p}\left(  s\right)  h=P\left(  s\right)
h$ when $h\in\mathcal{S}$ and $\left\langle P\left(  s\right)  h
,\psi\right\rangle =\left\langle h,P\left(  s\right)  \psi\right\rangle $
when $h,\psi\in\mathcal{S}$, hence (\ref{duality}) holds for all $h,\psi\in\mathcal{S}$, therefore for all $h\in H_{p}^{-\beta}  $ by density. Hence, from identity (\ref{weak-mild}) we get
\begin{align*}
\left\langle u\left(  t\right)  ,\psi\right\rangle &=\int_{t}^{T}\left\langle
P_{p}\left(  r-t\right)  b\left(  r\right)  \cdot\nabla u\left(  r\right)
,\psi\right\rangle\mathrm dr\\
&+\int_{t}^{T}\left\langle P_{p}\left(  r-t\right)
\left(  b\left(  r\right)  -\lambda u\left(  r\right)  \right)  ,\psi
\right\rangle\mathrm dr.
\end{align*}
This implies (\ref{mild in Sobolev}).
\end{proof}

\begin{theorem}\label{thm: uniqueness of K solution}
The solution $u$ of \eqref{eq: backwards PDE} is unique, in the sense that 
for each $\kappa_1, \kappa_2\in K (\beta,q), $ given two mild solutions 
 $u^{\kappa_1}, u^{\kappa_2}$  of \eqref{eq: backwards PDE},
 there exists $\kappa_0 = (\delta_0, p_0) \in K(\beta,q)$ such that $u^{\kappa_1}, u^{\kappa_2} \in C([0,T];H^{1+\delta_0}_{p_0})$ and the two solutions coincide in this bigger space.
\end{theorem}
\begin{proof}
In order to find a suitable $\kappa_0$ we proceed in two steps.

\textbf{Step 1.} Assume first that $p_1=p_2=:p$. Then $H^{\delta_i}_{p_i}\subset H_p^{\delta_1\wedge \delta_2}$. The intuition in Figure \ref{fig: the set K} is that we move downwards along the vertical line passing from $\frac1p$.

\textbf{Step 2.} If, on the contrary, $\frac{1}{p_1}<\frac{1}{p_2}$ (the opposite case is analogous) we may reduce ourselves to Step 1 in the following way: $H^{\delta_2 }_{p_2}\subset H^x_{p_1}$ for $x= \delta_2 - \frac{d}{p_2} +\frac{d}{p_1}$ (using  Theorem \ref{thm: fractional Morrey ineq}, equation \eqref{eq: embedding property}). Now $H^x_{p_1}$ and $H^{\delta_1}_{p_1}$ can be compared as in Step 1. The intuition in Figure \ref{fig: the set K} is that we move the rightmost point to the left along the line with slope $d$.

By Theorem \ref{thm: contraction property of integral operator} we have a unique mild solution $u^{\kappa_i}$ in $ C ([0,T]; H^{1+\delta_i}_{p_i})$ for each set of parameters $\kappa_i=(\delta_i,p_i)\in K(\beta,q)$, $i=0,1,2$. By Steps 1 and 2, the space with $i=0$ includes the other two, thus $u^{\kappa_i} \in  C ([0,T]; H^{1+\delta_0}_{p_0})$ for each $i=0,1,2$ and moreover $u^{\kappa_i}$ are mild solutions in $\mathcal S'$.  Lemma \ref{lm: mild solutions coincide in S'} concludes the proof.
\end{proof}

\subsection{Further regularity properties} We derive now stronger regularity properties for the mild solution $v$ of \eqref{eq: mild solution to fwd PDE}. Since $v(t,x) = u(T-t, x)$ the same properties hold for the mild solution $u$ of \eqref{mild in Sobolev}.

In the following lemma we show that the mild solution $v$ is differentiable in space and its gradient can be bounded by $\frac12$ for some $\lambda$ big enough. For this reason here we stress the dependence of the solution $v$ on the parameter $\lambda$ by writing $v_\lambda$.

\begin{lemma}\label{lm: bound for gradient u}
Let $(\delta, p)\in K(\beta,q)$ and let $v_\lambda$ be the mild solution to 
\eqref{eq: fwd PDE}
 in $H^{1+\delta}_p$. Fix $\rho$ such that the integral operator \eqref{eq: integral operator} is a contraction on $C([0,T];H^{1+\delta}_p)$ with the norm \eqref{eq: equivalent norm} and let $\lambda> \rho$. Then $v_\lambda(t)\in C^{1,\alpha}$ with $\alpha= \delta-d/p$  for each fixed $t$ and
\begin{align} \label{eq: bound for v}
&{ 
\sup_{0\leq t\leq T}\left( \sup_{x\in \mathbb R^d}  |  v_\lambda(t,x)| \right) \leq C, }
\\
\label{eq: bound for grad v}
&\sup_{0\leq t\leq T}\left( \sup_{x\in \mathbb R^d}  |\nabla v_\lambda(t,x)| \right)  \leq \frac{ c \|b\|_{\infty, H^{-\beta}_p} \lambda^{\frac{\delta+\beta-1}{2}}}{1 - c'  \| b\|_{\infty, H^{-\beta}_q} \lambda^{\frac{\delta+\beta-1}{2}} } ,
\end{align}
for some universal constants $C, c, c'$. In particular,
\[
\sup_{(t,x)\in [0,T]\times\mathbb R^d} |\nabla v_\lambda(t,x)| \to 0,
\]
as $\lambda \to \infty$.
\end{lemma}

\begin{proof}
{ By Theorem \ref{thm: fractional Morrey ineq} and the definition of the set $K(\beta,q)$ we have that $v_\lambda(t)\in C^{1,\alpha}$ and \eqref{eq: bound for v} holds using the definition of the norms in $ C ([0,T]; H^{1+\delta}_{p})$ and  $C^{1,\alpha}$. }
\\
Lemma \ref{lm: semigroup from -beta to 1+delta} ensures that $P_t w\in  H^{1+\delta}_p $ for  $w\in H^{-\beta}_p$ and so $\nabla P_t w\in  H^{\delta}_p$. By the fractional Morrey inequality  (Theorem \ref{thm: fractional Morrey ineq}) we have
  $P_tw\in C^{1,\alpha}(\mathbb R^d)$ and for each $t>0$
\begin{equation}\label{eq: Morrey applied to Pt w}
\sup_{x\in \mathbb R^d} |\left( \nabla P_t w \right)(x)| \leq c \|\nabla P_t w\|_{H^\delta_p} \leq c \| P_t w\|_{H^{1+\delta}_p} \leq c t^{-\frac{1+\delta+\beta}{2}} \| w\|_{H^{-\beta}_p},
\end{equation}
having used \eqref{eq: semigroup from -beta to 1+delta} in the latter inequality. Notice that the constant $c$ depends only on $\delta, p$ and $d$.

If we assume for a moment that the mild solution $v_\lambda$ of  \eqref{eq: fwd PDE}
  is also a solution of
\begin{align}\label{eq: mild form for v with e^lambda}
v_\lambda=&\int_0^t \mathrm e^{-\lambda (t-r)}P_{p}(t-r)\left(  b(r)\cdot\nabla v_\lambda(r)\right) \mathrm dr \\ \nonumber
&+\int_0^t \mathrm e^{-\lambda (t-r)}P_{p}(t-r) b(r) \mathrm dr,
\end{align}
then differentiating in $x$ we get
\begin{align*}
 \nabla v_\lambda(t, \cdot)= &\int_0^t \mathrm e^{-\lambda (t-r)} \nabla P_{p}(t-r) \left( b(r)\cdot\nabla v_\lambda(r) \right)\mathrm dr \\
  &+\int_0^t \mathrm e^{-\lambda (t-r)} \nabla P_{p}(t-r) b(r) \mathrm dr.
\end{align*}
We take the ${H^{\delta}_p}$-norm and use  \eqref{eq: Morrey applied to Pt w} with Lemma \ref{lm: pointwise product} to obtain
\begin{align*}
 \|\nabla v_\lambda(t)\|_{H^{\delta}_p}
\leq & c \int_0^t \mathrm e^{-\lambda (t-r)}(t-r)^{-\frac{1+\delta+\beta}{2}} \| b(r)\|_{H^{-\beta}_q} \|\nabla v_\lambda(r)\|_{H^{\delta}_p}  \mathrm dr \\
 &+c \int_0^t \mathrm e^{-\lambda (t-r)} (t-r)^{-\frac{1+\delta+\beta}{2}} \|b(r)\|_{H^{-\beta}_p} \mathrm dr \\
\leq & c'  \| b\|_{\infty, H^{-\beta}_q} \sup_{0\leq r\leq t}  \|\nabla v_\lambda(r)\|_{H^{\delta}_p} \int_0^t \mathrm e^{-\lambda (t-r)}(t-r)^{-\frac{1+\delta+\beta}{2}} \mathrm dr \\
 &+c\|b\|_{\infty, H^{-\beta}_p}  \int_0^t \mathrm e^{-\lambda (t-r)} (t-r)^{-\frac{1+\delta+\beta}{2}} \mathrm dr,
\end{align*}
so that by Lemma \ref{lm: stima integrale con fz gamma} we get
\begin{align*}
\sup_{0\leq t\leq T} \|\nabla v_\lambda(t)\|_{H^{\delta}_p} \leq& c'  \| b\|_{\infty, H^{-\beta}_q} \sup_{0\leq t\leq T}  \|\nabla v_\lambda(t)\|_{H^{\delta}_p} \lambda^{\frac{\delta+\beta-1}{2}} \\
&+ c\|b\|_{\infty, H^{-\beta}_p} \lambda^{\frac{\delta+\beta-1}{2}}.
\end{align*}
Choosing $\lambda > \lambda^*:=\left( \frac{1}{c'\| b\|_{\infty,H^{-\beta}_q}}\right)^{\frac{2}{\delta+\beta-1}} $ yields
\begin{align*}
\sup_{0\leq t\leq T} \|\nabla v_\lambda(t)\|_{H^{\delta}_p}
\leq \frac{ c\|b\|_{\infty, H^{-\beta}_p} \lambda^{\frac{\delta+\beta-1}{2}}}{1 - c'  \| b\|_{\infty, H^{-\beta}_q} \lambda^{\frac{\delta+\beta-1}{2}} },
\end{align*}
which tends to zero as $\lambda \to \infty$. The fractional Morrey inequality \eqref{eq: fractional Morrey ineq} together with the latter bound gives
\begin{align*}
\sup_{0\leq t\leq T}\left( \sup_{x\in \mathbb R^d}  |\nabla v_\lambda(t,x)| \right)  &\leq \sup_{0\leq t\leq T}  c \|\nabla v_\lambda(t)\|_{H^{\delta}_p}\\ 
& \leq \frac{ c\|b\|_{\infty, H^{-\beta}_p} \lambda^{\frac{\delta+\beta-1}{2}}}{1 - c'  \| b\|_{\infty, H^{-\beta}_q} \lambda^{\frac{\delta+\beta-1}{2}} },
\end{align*}
which tends to zero as $\lambda\to \infty$.

It is left to prove that a solution of  \eqref{eq: mild solution to fwd PDE} in $H_{p}^{1+\delta}$ it is also a solution of \eqref{eq: mild form for v with e^lambda}.
 There are several proofs of this fact, let us see
one of them. Computing each term against a test function $\psi\in\mathcal{S}$
we get the mild formulation
\begin{align*}
\left\langle v\left(  t\right)  ,\psi\right\rangle =&\int_{0}^{t}\left\langle
b\left(  r\right)  \cdot\nabla v\left(  r\right)  ,P\left(  t-r\right)
\psi\right\rangle\mathrm dr\\
&+\int_{0}^{t}\left\langle b\left(  r\right)  -\lambda
v\left(  r\right)  ,P\left(  t-r\right)  \psi\right\rangle\mathrm dr,
\end{align*}
used in the definition of mild solution in $\mathcal{S}^{\prime}$. Let us
choose in particular $\psi=\psi_{k}$ where $\psi_{k}\left(  x\right)
=e^{ix\cdot k}$, for a generic $k\in\mathbb{R}^{d}$, and let us write
$v_{k}\left(  t\right)  =\left\langle v\left(  t\right)  ,e^{ix\cdot
k}\right\rangle $ (the fact that $\psi_{k}$ is complex-valued makes no
difference, it is sufficient to treat separately the real and imaginary part).
Using the explicit formula for $P\left(  t\right)  $, it is not difficult to
check that
\begin{equation}\label{Fourier identity}%
P\left(  t\right)  \psi_{k}=e^{-\left(  \left\vert k\right\vert ^{2}+1\right)
t}\psi_{k}
\end{equation}
and therefore
\[
v_{k}\left(  t\right)  =\int_{0}^{t}e^{-\left(  \left\vert k\right\vert
^{2}+1\right)  \left(  t-r\right)  }g_{k}\left(  r\right) \mathrm dr-\lambda\int%
_{0}^{t}e^{-\left(  \left\vert k\right\vert ^{2}+1\right)  \left(  t-r\right)
}v_{k}\left(  r\right) \mathrm dr,
\]
where $g_{k}\left(  r\right)  =\left\langle b\left(  r\right)  \cdot\nabla
v\left(  r\right)  +b\left(  r\right)  ,\psi_{k}\right\rangle $. At the level
of this scalar equation it is an easy manipulation to differentiate and
rewrite it as%
\[
v_{k}\left(  t\right)  =\int_{0}^{t}e^{-\left(  \left\vert k\right\vert
^{2}+1+\lambda\right)  \left(  t-r\right)  }g_{k}\left(  r\right)\mathrm  dr.
\]
This identity, using again (\ref{Fourier identity}), can be rewritten as%
\[
\left\langle v\left(  t\right)  ,\psi_{k}\right\rangle =\int_{0}%
^{t}e^{-\lambda\left(  t-r\right)  }\left\langle b\left(  r\right)
\cdot\nabla v\left(  r\right)  +b\left(  r\right)  ,P\left(  t-r\right)
\psi_{k}\right\rangle dr
\]
and then we deduce  \eqref{eq: mild form for v with e^lambda} as we did in the proof of Lemma \ref{lm: mild solutions coincide in S'}.
\end{proof}

\begin{lemma}\label{lm: u jointly continuous}
Let $v=v_\lambda$ for $\lambda$ as in Lemma \ref{lm: bound for gradient u}. Then $v$ and $\nabla v$ are jointly continuous in $(t,x)$.
\end{lemma}
\begin{proof}
It is sufficient to prove the claim for $\nabla v$. Let $(t,x), (s,y) \in [0,T]\times \mathbb R^d $. We have
\begin{align*}
 | \nabla v(t,x) - \nabla v(s,y) | \leq&  | \nabla v(t,x) -\nabla v(s,x)| + |\nabla v(s,x) -\nabla v(s,y)|\\
 \leq & \sup_{x\in \mathbb R^d} |\nabla v(t,x) - \nabla v(s,x)| + |\nabla v(s,x) - \nabla v(s,y)|\\
 \leq &   \| v(t,\cdot) -v(s,\cdot)\|_{C^{1,\alpha}} +  \|v(s,\cdot)\|_{C^{1,\alpha}} |x-y|^\alpha \\
 \leq &   \| v(t,\cdot) -v(s,\cdot)\|_{H^{1+\delta}_p} +  \|v(s,\cdot)\|_{H^{1+\delta}_p} |x-y|^\alpha \\
 \leq &   \| v(t,\cdot) -v(s,\cdot)\|_{H^{1+\delta}_p} +  \|v\|_{C^\gamma([0,T]; H^{1+\delta}_p)} |x-y|^\alpha \\
 \leq & \|v\|_{C^\gamma([0,T]; H^{1+\delta}_p)} (|t-s|^\gamma +|x-y|^\alpha),
\end{align*}
having used the embedding property \eqref{eq: fractional Morrey ineq} with $\alpha= \delta - d/p$ and the H\"older property of $v$ from Lemma \ref{lm: bound for gradient u}.
\end{proof}

\begin{lemma}\label{lm: phi is invertible}
 For $\lambda$ large enough the function $x\mapsto \varphi(t,x)$ defined as $\varphi(t,x)= x+u(t,x)$ is invertible for each fixed $t\in [0,T]$ and, denoting its inverse by  $\psi(t, \cdot)$ the function $(t,y)\mapsto \psi(t, y)$ is jointly continuous. Moreover $\psi(t, \cdot)$ is Lipschitz with Lipschitz constant $k=2$, for every $t\in[0,T]$.
\end{lemma}
We will sometimes use the shorthand notation $\varphi_t$ for $\varphi(t, \cdot) $ and analogously for its inverse.

\begin{proof}
\textbf{Step 1} (invertibility of $\varphi_t$). Let $t$ be fixed and $x_1,x_2\in \mathbb R$. Recall that  by Lemma \ref{lm: bound for gradient u} for $\lambda$ large enough  we have
\begin{equation}\label{eq: bound 1/2 for nabla u}
\sup_{(t,x)\in [0,T]\times\mathbb R^d } | \nabla u(t,x) | \leq \frac12,
\end{equation}
so that
\[
 |u(t, x_2)- u(t,x_1)|\leq \int_0^1 | \nabla u(t, ax_2+(1-a)x_1)| |x_1-x_2|\mathrm d a \leq \frac12  |x_1-x_2|.
\]
Then the map $x\mapsto y-u(t,x)$ is a contraction  for each $y\in \mathbb R^d$ and therefore for each $y\in \mathbb R^d$ there exists a unique $x\in \mathbb R^d$ such that $x = y-u(t,x)$ that is $y= \varphi (t, x)$. Thus $\varphi(t, \cdot) $ is invertible for each $t\in [0,T]$ with inverse denoted by $\psi_t$.

\textbf{Step 2} (Lipschitz character of $\psi_t$, uniformly in $t$). To show that $\psi_t$ is Lipschitz with constant $k$ we can equivalently show that for each $ x_1,x_2\in \mathbb R^d$ it holds $|\varphi_t(x_1)-\varphi_t(x_2)|\geq \frac1k |x_1-x_2|$. We have
\begin{align*}
 |\varphi_t(x_1)-\varphi_t(x_2)| &\geq \inf_{x\in \mathbb R^d} |\nabla \varphi (t, x)| |x_1-x_2| = \frac12 |x_1-x_2|,
\end{align*}
because of \eqref{eq: bound 1/2 for nabla u} together with $\nabla \varphi = \mathrm I_d+\nabla u$.

\textbf{Step 3}  (continuity of $s\mapsto \psi(s, y)$). Let us fix $y\in \mathbb R^d$ and take $t_1, t_2\in[0,T]$. Denote by $x_1=\psi(t_1, y)$ and $x_2=\psi(t_2, y)$ so that $y= \varphi(t_1,x_1)= x_1+ u(t_1, x_1)$  and  $y= \varphi(t_2,x_2)= x_2+ u(t_2, x_2)$. We have
\begin{align} \label{eq: bound for phi-1(t)}
|\psi(t_1, y)-\psi(t_2, y) | &=  |x_1-x_2|  \nonumber \\
&= |u(t_1, x_1)-u(t_2, x_2) | \nonumber \\
& \leq |u(t_1, x_1)-u(t_1, x_2) |+ |u(t_1, x_2)-u(t_2, x_2) | \\
& \leq \frac12 |x_1-x_2| + |u(t_1, x_2)-u(t_2, x_2) | .\nonumber
\end{align}

Let us denote by $w(x):= u(t_1, x)-u(t_2, x)$. Clearly $w\in H^{1+\delta}_p$ for each $t_1,t_2$ and by Theorem \ref{thm: fractional Morrey ineq} (Morrey inequality) we have that $w$ is continuous, bounded and
\[
|u(t_1, x_2)-u(t_2, x_2) | \leq \sup_{x\in \mathbb R^d} |w( x)| \leq c \|w\|_{H^{1+\delta}_p}.
\]
By Theorem \ref{thm: contraction property of integral operator} $u\in C^\gamma ([0,T]; H^{1+\delta}_p)$ and so $ \|w\|_{H^{1+\delta}_p} \leq c |t_1-t_2|^\gamma$. Using this result together with \eqref{eq: bound for phi-1(t)} we obtain
\[
\frac12 |x_1-x_2| = \frac12 |\psi(t_1, y)-\psi(t_2, y) | \leq c |t_1-t_2|^\gamma,
\]
which shows the claim.\\
Continuity of $(t,y)\mapsto \psi(t, y)$ now follows.
\end{proof}

\begin{lemma}\label{lm: continuity of u wrt approximation}
 If $b_n\to b$ in $L^\infty\left ([0,T]; H^{-\beta}_{\tilde q, q} \right)$ then $v_n\to v$ in $ C([0,T];H^{1+\delta}_{p})$.
\end{lemma}

\begin{proof}
Let $\lambda>0$ be fixed. We consider the integral equation \eqref{eq: mild solution to fwd PDE} on $H^{1+\delta}_p$  so  the semigroup will be denoted by $P_p$.  Observe that by Lemma \ref{lm: semigroup from -beta to 1+delta} we have
\begin{align*}
\| P_p(t-r) & \left( b_n(r)\cdot \nabla v_n(r) -  b(r)\cdot \nabla v(r) \right) \|_{H^{1+\delta}_p}  \\
\leq &c (t-r)^{-\frac{1+\delta+\beta}{2}}  \left\| b_n(r)\cdot \nabla v_n(r) -  b(r)\cdot \nabla v(r)  \right\|_{H^{-\beta}_p}\\
\leq & c (t-r)^{-\frac{1+\delta+\beta}{2}}  \bigg( \|b_n(r)\|_{H^{-\beta}_q} \| v_n(r) - v(r)\|_{H^{1+\delta}_p} \\
&+ \|b_n(r) - b(r)\|_{H^{-\beta}_q} \|v(r)\|_{H^{1+\delta}_p}  \bigg)\\
\leq & c (t-r)^{-\frac{1+\delta+\beta}{2}}  \bigg( \|b_n\|_{\infty, H^{-\beta}_q} \| v_n(r) - v(r)\|_{H^{1+\delta}_p} \\
&+ \|b_n - b\|_{\infty, H^{-\beta}_q} \|v(r)\|_{H^{1+\delta}_p}  \bigg),
\end{align*}
where the second to last line is bounded through Lemma \ref{lm: pointwise product}. Thus, by \eqref{eq: mild solution to fwd PDE}
\begin{align*}
\| &v- v_n\|^{(\rho)}_{\infty, H^{1+\delta}_p}= \sup_{0\leq t\leq T } \mathrm e^{-\rho t}\|v(t)-v_n(t)\|_{H^{1+\delta}_p} \\
\leq &  \sup_{0\leq t\leq T } \mathrm e^{-\rho t}\bigg ( \int_0^t   \left\| P_p(t-r) \left( b_n(r)\cdot \nabla v_n(r) - b(r)\cdot\nabla v(r) \right) \right\|_{H^{1+\delta}_p} \mathrm d r \\
+ &  \int_0^t   \left\| P_p(t-r) \big( b_n(r)  -  b(r)  + \lambda (v(r)- v_n(r))\big) \right\|_{H^{1+\delta}_p} \mathrm d r \bigg )\\
\leq &  \sup_{0\leq t\leq T } \mathrm e^{-\rho t} \bigg( c  \|b_n\|_{\infty, H^{-\beta}_q} \int_0^t  (t-r)^{-\frac{1+\delta+\beta}{2}} \| v_n(r) - v(r)\|_{H^{1+\delta}_p}  \mathrm d r \\
&+ c  \|b_n - b \|_{\infty, H^{-\beta}_q} \int_0^t  (t-r)^{-\frac{1+\delta+\beta}{2}} \|  v(r)\|_{H^{1+\delta}_p}  \mathrm d r \\
&+  c  \|b_n - b \|_{\infty, H^{-\beta}_q} \int_0^t    (t-r)^{-\frac{1+\delta+\beta}{2}}  \mathrm d r
+ c\lambda \int_0^t \|v(r)- v_n(r)\|_{H^{1+\delta}_p} \mathrm dr \bigg )\\
\leq & c  \|b_n\|_{\infty, H^{-\beta}_q}   \sup_{0\leq t\leq T }   \int_0^t  \mathrm e^{-\rho( t-r) } (t-r)^{-\frac{1+\delta+\beta}{2}}  \mathrm e^{-\rho r } \| v_n(r) - v(r)\|_{H^{1+\delta}_p}  \mathrm d r \\
&+ c  \|b_n - b \|_{\infty, H^{-\beta}_q} \cdot\\
  &\phantom{space}\cdot\sup_{0\leq t\leq T } \int_0^t \mathrm e^{-\rho( t-r) }  (t-r)^{-\frac{1+\delta+\beta}{2}}   \mathrm e^{-\rho r } \left( \|v(r)\|_{H^{1+\delta}_p}+1\right)  \mathrm d r\\
&+ c  \lambda   \sup_{0\leq t\leq T }   \int_0^t  \mathrm e^{-\rho( t-r) } \mathrm e^{-\rho r } \| v_n(r) - v(r)\|_{H^{1+\delta}_p}  \mathrm d r,
\end{align*}
where we have used again Lemma \ref{lm: semigroup from -beta to 1+delta}. Consequently
\begin{align*}
 \| v- v_n\|^{(\rho)}_{\infty, H^{1+\delta}_p} \leq & c\|b_n\|_{\infty, H^{-\beta}_q}  \| v_n - v \|^{(\rho)}_{\infty, H^{1+\delta}_p} \rho^{\frac{\delta+\beta-1}{2}} \\
 &+ c  \|b_n - b \|_{\infty, H^{-\beta}_q} \left( \|v\|^{(\rho)}_{\infty, H^{1+\delta}_p}+ 1 \right ) \rho^{\frac{\delta+\beta-1}{2}}\\
 &+ c \lambda   \| v_n - v \|^{(\rho)}_{\infty, H^{1+\delta}_p} \rho^{-1} .
\end{align*}
The last bound is due to Lemma \ref{lm: stima integrale con fz gamma}. Since $\|b_n\|_{\infty, H^{-\beta}_q}\to \|b\|_{\infty, H^{-\beta}_q}$ then there exists $n_0\in \mathbb N$ such that $\|b_n\|_{\infty, H^{-\beta}_q}\leq 2 \|b\|_{\infty, H^{-\beta}_q}$ for all $n\geq n_0$. Choose now $\rho$ big enough in order to have
\[
1 - c \left(\|b\|_{\infty,H^{-\beta}_q} \rho^{\frac{\delta+\beta-1}{2}} + \lambda \rho^{-1} \right)>0
\]
 and then we have for each $n\geq n_0$
\[
 \| v- v_n\|^{(\rho)}_{\infty, H^{1+\delta}_p} \leq  c \frac{\left( \|v\|_{\infty, H^{1+\delta}_p}^{(\rho)} +1\right) \rho^{\frac{\delta+\beta-1}{2}} }{1 - c \left(\|b\|_{\infty,H^{-\beta}_q} \rho^{\frac{\delta+\beta-1}{2}} + \lambda \rho^{-1} \right) }  \|b_n - b \|_{\infty,H^{-\beta}_q},
\]
which concludes the proof.
\end{proof}

\begin{lemma} \label{nuovo lemma}
\begin{itemize}
 \item[(i)]  Let $\|b_n\|_{\infty, H^{-\beta}_{\tilde q,q}} \leq c \|b\|_{\infty, H^{-\beta}_{\tilde q,q}} $ for a constant $c$ not depending on $n$. Then there exists a constant $C>0$ such that
\[
\sup_{(t,x)\in [0,T]\times \mathbb R^d} \left (\left\vert u_{n}(t,x) \right\vert+\left\vert \nabla u_{n}(t,x)\right\vert \right) \leq C,
\]
for every $n\in\mathbb{N}$.
\item[(ii)] There exists $\lambda\geq0$ such that
\begin{equation}\label{eq: uniform bound nabla u}
\sup_{(t,x)\in [0,T]\times \mathbb R^d} \left\vert \nabla u_{n}(t,x)\right\vert \leq\frac{1}{2}
\end{equation}
and
\[
\sup_{(t,y)\in [0,T]\times \mathbb R^d} \left\vert \nabla\psi_n(t,y)\right\vert \leq2,
\]
for every $n\in\mathbb{N}$ and similarly for $\nabla u$ and $\nabla
\psi$.
\item[(iii)] If $b_{n}\to b$ in $L^\infty ([0,T];H^{-\beta}_{\tilde q,q})$, then we have
$u_{n}\rightarrow u$, $\nabla u_{n}\rightarrow\nabla u$, $\varphi
_{n}\rightarrow\varphi$ and $\psi_n\rightarrow\psi$
uniformly on $\left[  0,T\right]  \times\mathbb{R}^{d}$.

\end{itemize}
\end{lemma}

\begin{proof}
(i) The proof has the same structure as the proof of Lemma \ref{lm: continuity of u wrt approximation}, but slightly simplified as the difference $v_n -v$ is replaced with $v_n$. In the following bounds the constant $c$ may change from line to line and one gets
\begin{align*}
\|v_n\|^{(\rho)}_{\infty, H^{1+\delta}_p}\leq& c \|b_n\|_{\infty, H^{-\beta}_{\tilde q, q}} \|v_n\|^{(\rho)}_{\infty, H^{1+\delta}_p} \rho^{\frac{\delta+\beta-1}{2}}  \\
& +  c \|b_n\|_{\infty, H^{-\beta}_{q}} \rho^{\frac{\delta+\beta-1}{2}}  + c \lambda \|v_n\|^{(\rho)}_{\infty, H^{1+\delta}_p} \rho^{-1}  \\
\leq& c \|b\|_{\infty, H^{-\beta}_{\tilde q, q}} \|v_n\|^{(\rho)}_{\infty, H^{1+\delta}_p} \rho^{\frac{\delta+\beta-1}{2}}  \\
& +  c \|b\|_{\infty, H^{-\beta}_{\tilde q, q}} \rho^{\frac{\delta+\beta-1}{2}} + c \lambda \|v_n\|^{(\rho)}_{\infty, H^{1+\delta}_p} \rho^{-1} ,
\end{align*}
where the latter bound holds thanks to the assumption on the $b_n$'s. Now we choose $\rho$ large enough such that 
\[
1- c\left( \|b\|_{\infty, H^{-\beta}_{\tilde q, q}} \rho^{\frac{\delta+\beta-1}{2}} + \rho^{-1}  \right) >0
\]
and get for every $n\in \mathbb N$ 
\[
\|v_n\|^{(\rho)}_{\infty, H^{1+\delta}_p} \leq \frac{c \rho^{\frac{\delta+\beta-1}{2}} }{1- c \left( \|b\|_{\infty, H^{-\beta}_{\tilde q, q}} \rho^{\frac{\delta+\beta-1}{2}} + \rho^{-1}  \right)} \|b\|_{\infty, H^{-\beta}_{\tilde q, q}}=: C.
\]

(ii) The uniform bound \eqref{eq: uniform bound nabla u} on $\nabla u_n$ is obtained simply applying Lemma \ref{lm: bound for gradient u} to $u_{n}$ in place of $u_\lambda$. For what concerns the second bound involving $\nabla \psi_n$ we observe that $\nabla \varphi_n (t,x)$ is non-degenerate uniformly in $t,x,n$ since for each $\xi\in \mathbb R^d$ we have
\[
\vert\nabla \varphi_n(t,x) \cdot \xi\vert \geq |\xi| -|\nabla u_n (t,x)\cdot \xi| \geq \frac12 |\xi| ,
\]
having used \eqref{eq: uniform bound nabla u} for the latter inequality. This implies that $\nabla \psi_n (t,x)$ is well-defined for each $(t,x)$. Further note that by  Lemma \ref{lm: phi is invertible} we have that $\psi_n(t, \cdot)$ is Lipschitz with constant $k=2$, uniformly in $t$ and $n$, and this now implies the claim.

(iii) We know that $u_{n}\rightarrow u$ in $C\left(  \left[  0,T\right]
;H^{1+\delta}_p\right)  $, namely%
\[
\lim_{n\rightarrow\infty}\sup_{t\in\left[  0,T\right]  }\left\Vert
u_{n}(t)-u(t) \right\Vert _{H^{1+\delta}_p}=0.
\]
By Sobolev embedding theorem, there is a constant $C>0$ such that%
\begin{align*}
&\sup_{t\in\left[  0,T\right]  }\left(  \sup_{x\in \mathbb R^d}\left\vert u_{n}(t,x)-u(t,x)\right\vert 
+\sup_{x\in \mathbb R^n}\left\vert \nabla u_{n}(t,x)-\nabla u(t,x)\right\vert \right)  \\
&\leq C\sup_{t\in\left[  0,T\right]  }\left\Vert u_{n}(t)-u(t)\right\Vert _{H^{1+\delta}_p}.
\end{align*}
Hence $u_{n}\rightarrow u$ and $\nabla u_{n}\rightarrow\nabla u$, uniformly on
$\left[  0,T\right]  \times\mathbb{R}^{d}$. Since $\varphi_{n}-\varphi
=u_{n}-u$, we also have that $\varphi_{n}\rightarrow\varphi$ uniformly on
$\left[  0,T\right]  \times\mathbb{R}^{d}$. Let us prove the uniform
convergence of $\psi_n$ to $\psi$.

Given $y\in\mathbb{R}^{d}$, we know that for every $t\in\left[  0,T\right]  $
and $n\in\mathbb{N}$ there exist $x\left(  t\right)  ,x_{n}\left(  t\right)
\in\mathbb{R}^{d}$ such that
\begin{align*}
x\left(  t\right)  +u\left(  t,x\left(  t\right)  \right)    & =y\\
x_{n}\left(  t\right)  +u_{n}\left(  t,x_{n}\left(  t\right)  \right)    & =y
\end{align*}
and we have called $x\left(  t\right)  $ and $x_{n}\left(  t\right)  $ by
$\psi\left(  t,y\right)  $ and $\psi_n\left(  t,y\right)  $
respectively. Then%
\begin{align*}
\left\vert x_{n}\left(  t\right)  -x\left(  t\right)  \right\vert  
=&\left\vert u_{n}\left(  t,x_{n}\left(  t\right)  \right)  -u\left(
t,x\left(  t\right)  \right)  \right\vert \\
\leq&\left\vert u_{n}\left(  t,x_{n}\left(  t\right)  \right)  -u_{n}\left(
t,x\left(  t\right)  \right)  \right\vert +\left\vert u_{n}\left(  t,x\left(
t\right)  \right)  -u\left(  t,x\left(  t\right)  \right)  \right\vert \\
\leq &\sup_{(t,x) \in [0,T]\times \mathbb R^d}\left\vert \nabla u_{n}(t,x)\right\vert \left\vert x_{n}\left(
t\right)  -x\left(  t\right)  \right\vert \\
&+\sup_{(t,x) \in [0,T]\times \mathbb R^d}\left\vert u_{n}(t,x)-u(t,x)\right\vert.
\end{align*}
Since $\sup_{(t,x) \in [0,T]\times \mathbb R^d}\left\vert \nabla u_{n}(t,x)\right\vert \leq\frac{1}{2}$, we deduce
\[
\left\vert x_{n}\left(  t\right)  -x\left(  t\right)  \right\vert
\leq2 \sup_{(t,x) \in [0,T]\times \mathbb R^d} \left\vert u_{n}(t,x)-u(t,x)\right\vert, 
\]
 namely%
\[
\left\vert \psi_n\left(  t,y\right)  -\psi\left(
t,y\right)  \right\vert \leq2 \sup_{(t,x) \in [0,T]\times \mathbb R^d} \left\vert u_{n}(t,x)-u(t,x)\right\vert
\]
which implies that $\psi_n\rightarrow\psi$ uniformly on
$\left[  0,T\right]  \times\mathbb{R}^{d}$.
\end{proof}

\section{The virtual solution}\label{SVirtual}
From now on, we fix  $\lambda$ and $\rho$ big enough so that Theorem \ref{thm: contraction property of integral operator} and Lemma \ref{lm: phi is invertible} hold true. As usual, the drift $b$ is chosen according to Assumption \ref{assumption}.

\subsection{Heuristics and motivation}

 We consider the following $d$-dimensional SDE
\begin{equation}\label{eq: SDE with singular drift}
 \mathrm dX_t= b(t, X_t)\mathrm dt + \mathrm d W_t,\quad   t\in [0,T],
\end{equation}
with initial condition $X_0= x $ where $b$ is a distribution. Formally, the integral form is
\begin{equation}\label{eq: SDE with singular drift - integral form}
 X_t = x +  \int_0^t b(s,X_s) \mathrm ds  +W_t,\quad   t\in [0,T],
\end{equation}
but the integral appearing on the right hand side is not well-defined, a
priori. We aim to give a meaning to this equation by introducing a suitable
notion of solution to the SDE \eqref{eq: SDE with singular drift}. Let $u$ be a mild solution to the PDE \eqref{eq: backwards PDE}:\ we shall make use of $u$ to define a notion of
solution to the SDE \eqref{eq: SDE with singular drift}.

By stochastic basis we mean a pentuple $\left(  \Omega,\mathcal{F},\mathbb{F},P,W\right)  $ where $\left(  \Omega,\mathcal{F},P\right)  $ is a
complete probability space with a completed filtration $\mathbb{F=}\left(
\mathcal{F}_{t}\right)  _{t\in\left[  0,T\right]  }$ and $W$ is a
$d$-dimensional $\mathbb{F}$-Brownian motion. In the spirit of weak solutions,
we cannot assume that $\mathbb{F}$ is the completed filtration associated to $W$.

\begin{definition}\label{def: virtual solution} 
Given $x\in\mathbb{R}^{d}$, a \emph{virtual solution} to the SDE \eqref{eq: SDE with singular drift} with initial value $x$ is a stochastic basis
$\left(  \Omega,\mathcal{F},\mathbb{F},P,W\right)  $ and a continuous
stochastic process $X:=(X_{t})_{t\in\lbrack0,T]}$ on it, $\mathbb{F}$-adapted,
such that the integral equation
\begin{equation}\label{eq: virtual solution}
X_{t}=x+u(0,x)-u(t,X_{t})+(\lambda+1)\int_{0}^{t}u(s,X_{s})\mathrm{d}%
s+\int_{0}^{t} { ( \nabla u(s,X_{s}) + \mathrm I_d)} \mathrm{d}W_{s},
\end{equation}
holds for all $t\in[0,T]$, with probability one. Here $\mathrm I_d$ denotes the $d\times d $ identity matrix and $u$ is the unique
mild solution to the PDE \eqref{eq: backwards PDE}. We shorten the notation and say
that $\left(  X,\mathbb{F}\right)  $ is a virtual solution when the previous
objects exist with the required properties.
\end{definition}

The motivation for this definition comes from two facts: i) the
not-properly-defined expression $\int_{0}^{t}b(s,X_{s})\mathrm{d}s$ does not
appear in the formulation; ii) when $b$ is a function with reasonable
regularity, classical solutions of the SDE \eqref{eq: SDE with singular drift} are also virtual solutions;
 this is the content of  Proposition \ref{pro: classical solutions are virtual solutions}, where we will illustrate this fact by considering two examples, one of which is the class of drifts investigated by \cite{Kry-Ro}. Similar arguments can be developed for the bounded measurable drift considered by \cite{V}.

\subsection{Existence and uniqueness of the virtual solution}
To find a virtual solution $\left(  X,\mathbb{F}\right)  $ to \eqref{eq: SDE with singular drift} we first make the following observation. Let us assume that $\left(  X,\mathbb{F}\right)  $  is a virtual solution of \eqref{eq: SDE with singular drift} with initial value $X_0=x$ and let us introduce the transformation $\varphi(t, x):= x + u(t,x)$ and set $Y_t=\varphi(t, X_t)$ for $  t\in [0,T]$. From \eqref{eq: virtual solution} we obtain
\begin{equation*}
 \varphi(t,X_t) = x + u(0,x) + (\lambda+1) \int_0^t u(s,X_s) \mathrm ds  +\int_0^t {(\nabla u(s,X_s) +\mathrm I_d)} \mathrm dW_s.
\end{equation*}
Since the function $\varphi(t, \cdot)$ is invertible  for all $t\in[0,T]$, we can consider the SDE
\begin{equation}\label{eq: SDE for Yt}
 Y_t = y + (\lambda+1) \int_0^t u(s,\psi(s, Y_s)) \mathrm ds  +\int_0^t {(\nabla u(s,\psi(s, Y_s)) + \mathrm I_d)} \mathrm dW_s,
\end{equation}
for $t\in[0,T]$, where $y=x+u(0,x)$. Hence $\left(  Y,\mathbb{F}\right)  $ where $Y:=(Y_t)_{t\in[0,T]}$, is a solution of \eqref{eq: SDE for Yt} with initial value $y\in\mathbb R$. Conversely, if  $\left(  Y,\mathbb{F}\right)  $ is the solution of \eqref{eq: SDE for Yt} with initial value $y\in\mathbb R$, then $\left(  X,\mathbb{F}\right)  $ defined by  \[X_t = \psi(t, Y_t), \quad   t\in [0,T],\] will give us the virtual solution of the SDE \eqref{eq: SDE with singular drift} with distributional drift and with initial value $x=\psi(0,y)$.

As mentioned above, to gain a better understanding of the concept of virtual solution, we first compare it to some classical solutions. For example let us consider the class of drifts investigated by \cite{Kry-Ro}.  Let $b$ be a measurable function $b:\left[  0,T\right]
\times\mathbb{R}^{d}\rightarrow\mathbb{R}^{d}$ such that
\[
\int_{0}^{T}\left(  \int_{\mathbb{R}^{d}}\left\vert b\left(  t,x\right)
\right\vert ^{p}dx\right)  ^{q/p}\mathrm d t<\infty,
\]
(we say that $b\in L_{t}^{q}\left(  L_{x}^{p}\right)  $) for some $p,q\geq2$
such that
\[
\frac{d}{p}+\frac{2}{q}<1.
\]
Under this assumption, there exists a strong solution $\left(  X,\mathbb{F}\right)  $ to the SDE
\eqref{eq: SDE with singular drift} and it is pathwise unique, see \cite{Kry-Ro}.

\begin{proposition}\label{pro: classical solutions are virtual solutions} Suppose that one of the following conditions holds:  
\begin{itemize}
\item[(i)] $b \in C\left(  \left[  0,T\right]  ;C_{b}^{1}\left(  \mathbb{R}%
^{d};\mathbb{R}^{d}\right)  \right)  $ (bounded with bounded first derivatives);
\item[(ii)]  $b\in  L_{t}^{q}\left(  L_{x}^{p}\right)  $. 
\end{itemize}
Then the classical solution  $\left(  X,\mathbb{F}\right)  $ to the SDE \eqref{eq: SDE with singular drift} is also a virtual solution.
\end{proposition}

\begin{proof}

Suppose condition (i) holds. Let $u$ be the unique classical solution of equation \eqref{eq: backwards PDE};\ $u$ is (at least) of class $C^{1,2}\left(  \left[  0,T\right]  \times\mathbb{R}^{d};\mathbb{R}^{d}\right)$. Since $\varphi\left(  t,x\right)
=x+u\left(  t,x\right)   $ then  $\varphi\in C^{1,2}\left(  \left[
0,T\right]  \times\mathbb{R}^{d};\mathbb{R}^{d}\right)  $ as well. Let $X$ be
the unique strong solution of equation \eqref{eq: SDE with singular drift} and let $Y_{t}=\varphi \left(  t,X_{t}\right)  $. By It\^{o}'s formula, $Y$
satisfies equation \eqref{eq: SDE for Yt}  and thus $X_{t} =\psi\left(  t,Y_{t}\right)  $ is also a virtual solution. 

 Suppose now that condition (ii) holds.  The solution $u$ of the PDE \eqref{eq: backwards PDE}, when $b$ is of class $L_{t}^{q}\left(  L_{x}^{p}\right)  $ with the assumed constraints on $\left(q,p\right)  $, belongs to $L_{t}^{q}\left(  L_{x}^{p}\right)  $ with its first
and second spatial derivatives, the first spatial derivatives are continuous
and bounded, and other regularity properties hold; see \cite{Kry-Ro}. In
particular, it is proved there that It\^{o}'s formula extends to such functions
$u$ and we get
\begin{align*}
\mathrm{d}u(t,X_{t})= &  \left(  \frac{\partial u}{\partial t}(t,X_{t}%
)+\frac{1}{2}\Delta u(t,X_{t})+\nabla u(t,X_{t})b(t,X_{t})\right)
\mathrm{d}t\\
&  +\nabla u(t,X_{t})\mathrm{d}W_{t}\\
= &  (\lambda+1)u(t,X_{t})\mathrm{d}t-b(t,X_{t})\mathrm{d}t+\nabla
u(t,X_{t})\mathrm{d}W_{t}.
\end{align*}
The integral form of the last equation
\[
u(t,X_{t})=u(0,x)+(\lambda+1)\int_{0}^{t}u(s,X_{s})\mathrm{d}s-\int_{0}%
^{t}b(s,X_{s})\mathrm{d}s+\int_{0}^{t}\nabla u(s,X_{s})\mathrm{d}W_{s},
\]
allows us to evaluate the singular term $\int_{0}^{t}b(s,X_{s})\mathrm{d}s$
as
\[
\int_{0}^{t}b(s,X_{s})\mathrm{d}s=u(0,x)-u(t,X_{t})+(\lambda+1)\int_{0}%
^{t}u(s,X_{s})\mathrm{d}s+\int_{0}^{t}\nabla u(s,X_{s})\mathrm{d}W_{s}.
\]
This proves identity (\ref{eq: virtual solution}) and thus $\left(  X,\mathbb{F}\right)  $ is a virtual solution.
\end{proof}

\begin{proposition}\label{pr: existence and uniqueness weak solution Yt}
For every initial condition $y\in\mathbb R$ there exists a unique weak solution $\left(  Y,\mathbb{F}\right)  $ to the SDE \eqref{eq: SDE for Yt} with initial value $y$.
\end{proposition}
\begin{proof}
We know that $u, \nabla u$ and $\psi$ are jointly continuous in time and space 

by Lemma \ref{lm: u jointly continuous} and Lemma \ref{lm: phi is invertible}. This implies that the drift of $Y$
\[
\mu(t,y):= (\lambda+1)  u(t,\psi(t, y))
\]
and the diffusion coefficient
\[
\sigma(t,y):= \nabla u(t,\psi(t, y)) + \mathrm I_d = \nabla\varphi (t,\psi(t, y))
\]
are continuous. Since by  Lemma \ref{lm: bound for gradient u} the function $u$ and its gradient are uniformly bounded, we also have that $\mu $ and $\sigma $ are uniformly bounded. 
 Moreover $\sigma$ is uniformly non-degenerate since for all $x,\xi\in \mathbb R^d$ and $t\in [0,T]$
\begin{align*}
&|\sigma^T(t,x)\xi| = |\xi+ \xi \cdot \nabla u(t, \psi(t,y))| \\
&\geq |\xi|- |\xi \cdot \nabla u(t, \psi(t,y))| \geq \frac12 |\xi|
\end{align*}
by \eqref{eq: bound 1/2 for nabla u}.  Thus Theorem 10.2.2 in \cite{sv} yields existence and uniqueness of a weak solution of SDE \eqref{eq: SDE for Yt} for every initial value $y\in \mathbb R$.
\end{proof}

\begin{theorem}\label{thm: existence and uniqueness virtual solution Xt}
For every $x\in\mathbb R$ there exists a unique in law virtual solution $\left(  X,\mathbb{F}\right)  $ to the SDE \eqref{eq: SDE with singular drift} with initial value $X_0=x$ given by $X_t = \psi(t, Y_t), t\in [0,T]$, where $\left(  Y,\mathbb{F}\right)  $ is the solution with initial value $y=x+u(0,x)$ given in Proposition \ref{pr: existence and uniqueness weak solution Yt}.
\end{theorem}
\begin{proof}
We shorten the notation in the proof and write $X$ and $Y$ in place of $\left(  X,\mathbb{F}\right)  $ and $\left(  Y,\mathbb{F}\right)  $ respectively.\\
\emph{Existence.} Let us fix the initial condition $x\in \mathbb R$. By Proposition \ref{pr: existence and uniqueness weak solution Yt} there exists a unique solution $Y$ to the SDE \eqref{eq: SDE for Yt} with initial value $y= x+u(0,x)$. Let $X$ be defined by $X_t= \psi(t, Y_t), t\in [0,T]$. By construction $X$ is a virtual solution of \eqref{eq: SDE with singular drift} with initial condition $X_0=x$.\\
\emph{Uniqueness.} Suppose that $Z:=(Z_t)_{t\geq0}$  is another virtual solution to \eqref{eq: SDE with singular drift}.  Then $\tilde Y$ defined by  $\tilde Y_t= \varphi(t,Z_t), t\in[0,T]$ is a solution to \eqref{eq: SDE for Yt} with initial value $\tilde Y_0 = \varphi(0, x)= x+u(0,x)$. Since equation \eqref{eq: SDE for Yt} admits uniqueness in law, the law of $\tilde Y$ coincides with the law of $Y$ and by the invertibility of $\varphi(t,\cdot)$ for each $t\in [0,T]$ we get that the laws of $X$ and $Z$ coincide.
\end{proof}

\subsection{Virtual solution as limit of classical solutions}\label{SFinal}

The concept of \emph{virtual solution} is very convenient in order to prove weak
existence and uniqueness;\ however, it may look a bit artificial. Moreover, a
priori, the virtual solution may depend on the parameter $\lambda$. These
problems are solved by the next proposition which identifies the virtual
solution (for any $\lambda$) as the limit of classical solutions. This result
relates also to the concept of solution introduced in \cite{basschen2}.

\begin{proposition}\label{pr: limit of classical sol}
Let $b_{n}:\left[  0,T\right]  \times\mathbb{R}^{d}\rightarrow\mathbb{R}^{d}$
be vector fields such that

\begin{itemize}
\item[(i)] $b_{n}\in C\left(  \left[  0,T\right]  ;C_{b}^{1}\left(
\mathbb{R}^{d};\mathbb{R}^{d}\right)  \right)  $ (bounded with bounded first derivatives);

\item[(ii)] $b_{n}\rightarrow b$ in $L^{\infty}\left(  \left[  0,T\right]
;H_{\tilde{q},q}^{-\beta}\right)  $.
\end{itemize}

Then the unique strong solution to the equation%
\begin{equation}
dX_{t}^{n}=\mathrm{d}W_{t}+b_{n}\left(  t,X_{t}^{n}\right)  \mathrm{d}t,\qquad
X_{0}=x,\label{eq n}%
\end{equation}
converges in law to the virtual solution $\left(  X,\mathbb{F}\right)  $ of
equation \eqref{eq: SDE with singular drift}.
\end{proposition}

\begin{proof}
We shorten the notation in the proof and write $X$ and $Y$ in place of
$\left(  X,\mathbb{F}\right)  $ and $\left(  Y,\mathbb{F}\right)  $
respectively. Recall that $\lambda$ has been chosen big enough at the
beginning of the section.

\textbf{Step 1} ($X^{n}$ as virtual solutions). {Let $u_{n}$ be the unique
classical solution of equation \eqref{eq: backwards PDE} replacing $b$ with
$b_{n}$. By Proposition \ref{pro: classical solutions are virtual solutions},
part (i), we have that the unique strong solutions $X^{n}$ of equations
(\ref{eq n}) are also virtual solutions. Here $Y^{n}_t=\varphi_{n}\left(
t,X_{t}^{n}\right)  $ satisfies equation \eqref{eq:
SDE for Yt} with $b_{n}$ replacing $b$. Let us denote by $\tilde{b}_{n}$
and $\tilde{\sigma}_{n}$ (respectively $\tilde{b}$ and $\tilde{\sigma}$)
the drift and diffusion coefficient of the equation for $Y^{n}$ (respectively
$Y$), that is
\begin{align*}
&  \tilde{b}_{n}(t,x):=(\lambda+1)u_{n}(t,\psi_{n}(t,x)),\\
&  \tilde{\sigma}_{n}(t,x):=\nabla u_{n}(t,\psi_{n}(t,x))+\mathrm{I}_{d}.
\end{align*}
}

\textbf{Step 2} (Upper bounds on $u_{n}$ and $\nabla u_{n}$,
uniformly in $n$). {Since $b_{n}$ converges to $b$ in $L^{\infty
}([0,T];H_{\tilde{q},q}^{-\beta})$ there exists $n_{0}$ such that for each
$n\geq n_{0}$ we have $\Vert b_{n}\Vert_{\infty,H_{\tilde{q},q}^{-\beta}}%
\leq2\Vert b\Vert_{\infty,H_{\tilde{q},q}^{-\beta}}$ and so we can apply Lemma
\ref{nuovo lemma} (i) and find a constant $C_{1}>0$ such that
\begin{align*}
\left\vert u_{n}\left(  r,z\right)  \right\vert  &  \leq C_{1}\\
\left\vert \nabla u_{n}\left(  r,z\right)  +\mathrm{I}_{d}\right\vert  &  \leq
C_{1},
\end{align*}
for all $(r,z)\in\lbrack0,T]\times\mathbb{R}^{d}$.}

\textbf{Step 3} (Tightness of the laws of $Y^{n}$). {For $0\leq s\leq t\leq T$
we have
\begin{align*}
\left\vert Y_{t}^{n}-Y_{s}^{n}\right\vert ^{4} &  \leq8\left(  \lambda
+1\right)  ^{4}\left(  \int_{s}^{t}\left\vert u_{n}\left(  r,\psi_{n}\left(
r,Y_{r}^{n}\right)  \right)  \right\vert \mathrm{d}r\right)  ^{4}\\
&  +8\left\vert \int_{s}^{t}\left[  \nabla u_{n}\left(  r,\psi_{n}\left(
r,Y_{r}^{n}\right)  \right)  +\mathrm{I}_{d}\right]  \mathrm{d}W_{r}%
\right\vert ^{4}.
\end{align*}
By the result in Problem 3.29 and Remark 3.30 of \cite{ks} (see also Theorem
4.36 in \cite{da-prato_zabczyk92} given even in Hilbert spaces) there is a
constant $C_{2}>0$ such that%
\begin{align*}
E\left[  \left\vert Y_{t}^{n}-Y_{s}^{n}\right\vert ^{4}\right]   &
\leq8\left(  \lambda+1\right)  ^{4}C_{1}^{4}\left(  t-s\right)  ^{4}\\
&  +8C_{2}E\left[  \left(  \int_{s}^{t}\left\vert \nabla u_{n}\left(
r,\psi_{n}\left(  r,Y_{r}^{n}\right)  \right)  +\mathrm{I}_{d}\right\vert
^{2}\mathrm{d}r\right)  ^{2}\right]  \\
&  \leq8\left(  \lambda+1\right)  ^{4}C_{1}^{4}\left(  t-s\right)  ^{4}%
+8C_{1}^{4}C_{2}\left(  t-s\right)  ^{2}.
\end{align*}
This obviously implies%
\[
E\left[  \left\vert Y_{t}^{n}-Y_{s}^{n}\right\vert ^{4}\right]  \leq
C\left\vert t-s\right\vert ^{2},
\]
for some constant $C>0$ depending on $T$ and independent of $n$ }%
(recall that $\lambda$ is given).{ Moreover the initial
condition $y_{0}$ is real and independent of $n$, thus a tightness criterion
(see Corollary 16.9, in \cite{kal02}) implies the
tightness of the laws of $Y^{n}$ in $C\left(  \left[  0,T\right]
;\mathbb{R}^{d}\right)  $. }

\textbf{Step 4} (Weak convergence of $Y^{n}$ to a solution $Y$). Let $Y$ be
the process defined by $Y_{t}=\varphi\left(  t,X_{t}\right)  $, %
$t\in\left[  0,T\right]  $, where $X$ is the virtual solution in the
statement of the proposition. We want to prove that the laws of $Y^{n}$
converge weakly to the law of $Y$. Denote by $\mu^{n}$ and $\mu$ the laws of
$Y^{n}$ and $Y$, respectively. To show weak convergence of $\mu^{n}$ to $\mu$
it is enough to prove that, for any subsequence $\mu^{n_{k}}$, there exists a
further subsequence $\mu^{n_{k_{j}}}$ which converges weakly to $\mu$. Given
$\mu^{n_{k}}$, a converging subsequence $\mu^{n_{k_{j}}}$ exists, since
$\{\mu^{n}\}$ is tight by Step 3, so $\{\mu^{n_{k}}\}$ is also tight. Denote
by $\mu^{\prime}$ the weak limit of $\mu^{n_{k_{j}}}$ as $j\rightarrow\infty$.
If we prove that $\mu^{\prime}=\mu$, for any choice of the subsequence
$\mu^{n_{k}}$, then we have that the whole sequence $\mu^{n}$ converges to
$\mu$. 
Just to simplify notations, we shall denote $\mu^{n_{k_{j}}}$ by
$\mu^{n}$ and assume that $\mu^{n}$ converges weakly, as $n\rightarrow\infty$,
to $\mu^{\prime}$. 

Since $\mu^{n}\rightharpoonup\mu^{\prime}$, by Skorokhod's representation
theorem there exists a probability space $(\tilde{\Omega}%
,\tilde{\mathcal{F}},\tilde{\mathbb{P}})$ and random variables
$\tilde{Y}^{n}$ (resp. $\tilde{Y}$) taking values in 
$C([0,T];\mathbb{R}^{d})$ endowed with the Borel $\sigma$-field, with
laws $\mu^{n}$ (resp. $\mu^{\prime}$), such that $\tilde{Y}^{n}\rightarrow
\tilde{Y}$ in $C([0,T];\mathbb{R}^{d})$ a.s. If we prove
that $\tilde{Y}$ is a weak solution of equation \eqref{eq:
SDE for Yt}, since uniqueness in law holds for that equation and $Y$ is
another weak solution, we get $\mu^{\prime}=\mu$. To prove that $\tilde{Y}$ is
a weak solution of equation \eqref{eq:
SDE for Yt}, it is sufficient to prove (see Theorem 18.7 in \cite{kal02}) that
$\tilde{Y}$ solves the following martingale problem: for every $g\in
C^{\infty}(\mathbb{R}^{d};\mathbb{R})$ with compact support the
process %
\[
\tilde{M}_{t}:=g({\tilde{Y}}_{t})-g({\tilde{Y}}_{0})-\int_{0}^{t}%
L_{r}g({\tilde{Y}}_{r})\mathrm{d}r
\]
is a martingale, where{
\[
L_{r}:=\frac{1}{2}\sum_{i,j=1}^{d}\tilde{a}_{i,j}(r,\cdot)\frac{\partial^{2}%
}{\partial x_{i}\partial x_{j}}+\sum_{i=1}^{d}\tilde{b}_{i}(r,\cdot
)\frac{\partial}{\partial x_{i}}%
\]
}and $\tilde{a}(r,x):=\tilde{\sigma}(r,x)\tilde{\sigma}^{\ast}%
(r,x)$. Let $0\leq s<t\leq T$ and let $F:C([0,s];\mathbb{R}%
^{d})\rightarrow\mathbb{R}$ be a bounded continuous functional. To
prove that $\tilde M$ is a martingale it is enough to show that 
\begin{equation}
E\left[  \left(  \tilde{M}_{t}-\tilde{M}_{s}\right)  F({\tilde{Y}}%
_{r};r\leq s)\right]  =0\label{eq: expectation martingale y}%
\end{equation}
on the probability space $(\tilde{\Omega},\tilde{\mathcal{F}}%
,\tilde{\mathbb{P}})$.

Recall that $Y^{n}$ satisfies equation \eqref{eq: SDE for Yt} (with $b$
replaced by $b_{n}$). Hence, by It\^{o}'s formula,%
\[
M_{t}^{n}:=g({Y_{t}^{n}})-g({Y_{0}^{n}})-\int_{0}^{t}L_{r}^{n}g({Y_{r}^{n}%
})\mathrm{d}r
\]
is a martingale, where{
\[
L_{r}^{n}:=\frac{1}{2}\sum_{i,j=1}^{d}\tilde{a}_{i,j}^{n}(r,\cdot
)\frac{\partial^{2}}{\partial x_{i}\partial x_{j}}+\sum_{i=1}^{d}\left(
\tilde{b}_{n}\right)  _{i}(r,\cdot)\frac{\partial}{\partial x_{i}}%
\]
}and $\tilde{a}^{n}(r,x):=\tilde{\sigma}_{n}(r,x)\tilde{\sigma}_{n}^{\ast
}(r,x)$. Therefore
\[
E\left[  \left(  M_{t}^{n}-M_{s}^{n}\right)  F({Y}_{r}^{n};r\leq s)\right]
=0.
\]
This identity depends only on the law of $Y^{n}$, hence%
\begin{equation}
E\left[  \left(  \tilde{M}_{t}^{n}-\tilde{M}_{s}^{n}\right)
F({\tilde{Y}}_{t}^{n};r\leq s)\right]
=0.\label{eq: expectation martingale y^n}%
\end{equation}
Let us denote by $\tilde{{Z}}_{s,t}^{n}$ (resp. $\tilde{{Z}}_{s,t}$)
the random variable inside the expectation in
\eqref{eq: expectation martingale y^n} (resp. \eqref{eq:
expectation martingale y}).  Each factor in $\tilde{{Z}}_{s,t}^{n}$ (resp. $\tilde{{Z}%
}_{s,t}$) is uniformly bounded, by the boundedness of $F$, $g$ and its
derivatives, $\tilde{b}_{n}$ and $\tilde{\sigma}_{n}$ (due to boundedness of $u$ and $\nabla u$). Moreover $\tilde{{Z}}_{s,t}^{n}\rightarrow
\tilde{Z}_{s,t}$ a.s. since ${\tilde{Y}}^{n}(\omega)\rightarrow
{\tilde{Y}}(\omega)$ for almost every $\omega$ uniformly on compact sets, $F$, $g$ and its derivatives are continuous, $\tilde{b}_{n}\rightarrow\tilde{b}$
and $\tilde{\sigma}_{n}\rightarrow\tilde{\sigma}$ uniformly on compact sets
(remind that $\psi_{n}\rightarrow\psi,u^{n}\rightarrow u,\nabla u^{n}\rightarrow
\nabla u$ uniformly  by Lemma \ref{nuovo lemma} (iii)).
Therefore, by Lebesgue's dominated convergence theorem we conclude that
\eqref{eq: expectation martingale y} holds.

\textbf{Step 5} (Weak convergence of $X^{n}$ to $X$). The final step consists
in showing that $X^{n}$ converges to $X$ in law. Recall that $X^{n}=\psi
_{n}(\cdot,Y^{n})$.   By Lemma \ref{lm: phi is invertible}, $\psi$ is uniformly continuous on compact sets and $Y^{n}\rightarrow Y$ in law by Step 4; it is then easy to deduce that $\psi(\cdot,Y^{n})\rightarrow\psi(\cdot,Y)$ in law. Indeed given a continuous and bounded functional $F:C([0,T];\mathbb R^d )\to \mathbb R$, the functional $\eta \mapsto F(\psi(\cdot, \eta(\cdot))) $ is still a continuous bounded functional.  Moreover $\psi_{n}(\cdot,Y^{n})-\psi(\cdot,Y^{n})\rightarrow0$ in $C([0,T];\mathbb{R}%
^{d})$ $P$-a.s. since $\psi_{n}\rightarrow\psi$ uniformly by Lemma
\ref{nuovo lemma} (iii); hence $\psi_{n}(\cdot,Y^{n})- \psi(\cdot,Y^{n})\to 0$ in probability. Then by Theorem 4.1 in \cite[Section 4,
Chapter 1]{billingsley68} we have that $\psi_{n}(\cdot,Y^{n})\rightarrow
\psi(\cdot,Y)$ weakly, which concludes the proof.
\end{proof}

Examples of $b_n$ which verify (ii) in Proposition \ref{pr: limit of classical sol} are easily obtained by convolutions of $b$ against a sequence of mollifiers converging to a Dirac measure.

\vspace{15 pt}
{\bf Acknowledgments.} 
The authors wish to thank the anonymous referee for her/his careful revision which helped to clarify and improve considerably  the initial version of this paper.

\vspace{5pt}
The third named author benefited from the
 support of the ``FMJH Program Gaspard Monge in optimization and operation
 research'' (Project 2014-1607H).
The  second and third named authors were also partially supported
by the ANR Project MASTERIE 2010 BLAN 0121 01.

\bibliographystyle{plain}
\bibliography{SingularDrift-biblio}

\begin{thebibliography}{10}

\bibitem{basschen1}
R.~F. Bass and Z.-Q. Chen.
\newblock Stochastic differential equations for {D}irichlet processes.
\newblock {\em Probab. Theory Related Fields}, 121(3):422--446, 2001.

\bibitem{basschen2}
R.~F. Bass and Z.-Q. Chen.
\newblock Brownian motion with singular drift.
\newblock {\em Ann. Probab.}, 31(2):791--817, 2003.

\bibitem{billingsley68}
P.~Billingsley.
\newblock {\em Convergence of Probability measures}.
\newblock Wiley Series in Probability and Mathematical Statistics: Probability
  and Mathematical Statistics. John Wiley \& Sons, Inc., New York, second
  edition, 1968.

\bibitem{blei}
S.~Blei and H.~J. Engelbert.
\newblock One-dimensional stochastic differential equations with generalized
  and singular drift.
\newblock {\em Stochastic Process. Appl.}, 123(12):4337--4372, 2013.

\bibitem{da-prato_zabczyk92}
G.~Da~Prato and J.~Zabczyk.
\newblock {\em Stochastic equations in infinite dimensions}, volume~44 of {\em
  Encyclopedia of Mathematics and its Applications}.
\newblock Cambridge University Press, Cambridge, 1992.

\bibitem{d1}
E.~B. Davies.
\newblock {\em Heat kernels and spectral theory}, volume~92 of {\em Cambridge
  Tracts in Mathematics}.
\newblock Cambridge University Press, Cambridge, 1989.

\bibitem{esd}
H.-J. Engelbert and W.~Schmidt.
\newblock On one-dimensional stochastic differential equations with generalized
  drift.
\newblock In {\em Stochastic differential systems ({M}arseille-{L}uminy,
  1984)}, volume~69 of {\em Lecture Notes in Control and Inform. Sci.}, pages
  143--155. Springer, Berlin, 1985.

\bibitem{ew}
H.-J. Engelbert and J.~Wolf.
\newblock Strong {M}arkov local {D}irichlet processes and stochastic
  differential equations.
\newblock {\em Teor. Veroyatnost. i Primenen.}, 43(2):331--348, 1998.

\bibitem{priola}
F.~Flandoli, M.~Gubinelli, and E.~Priola.
\newblock Well-posedness of the transport equation by stochastic perturbation.
\newblock {\em Invent. Math.}, 180(1):1--53, 2010.

\bibitem{frw2}
F.~Flandoli, F.~Russo, and J.~Wolf.
\newblock Some {SDE}s with distributional drift. {I}. {G}eneral calculus.
\newblock {\em Osaka J. Math.}, 40(2):493--542, 2003.

\bibitem{frw1}
F.~Flandoli, F.~Russo, and J.~Wolf.
\newblock Some {SDE}s with distributional drift. {II}. {L}yons-{Z}heng
  structure, {I}t\^o's formula and semimartingale characterization.
\newblock {\em Random Oper. Stochastic Equations}, 12(2):145--184, 2004.

\bibitem{h1}
M.~Hinz and M.~Z{\"a}hle.
\newblock Gradient type noises. {II}. {S}ystems of stochastic partial
  differential equations.
\newblock {\em J. Funct. Anal.}, 256(10):3192--3235, 2009.

\bibitem{i1}
E.~Issoglio.
\newblock Transport equations with fractal noise---existence, uniqueness and
  regularity of the solution.
\newblock {\em Z. Anal. Anwend.}, 32(1):37--53, 2013.

\bibitem{kal02}
O.~Kallenberg.
\newblock {\em Foundation of {M}odern {P}robability}.
\newblock Springer, second edition, 2002.

\bibitem{ks}
I.~Karatzas and S.~E. Shreve.
\newblock {\em Brownian motion and stochastic calculus}, volume 113 of {\em
  Graduate Texts in Mathematics}.
\newblock Springer-Verlag, New York, second edition, 1991.

\bibitem{karatzas}
J.~Karatzas and I.~Ruf.
\newblock Pathwise solvability of stochastic integral equationswith generalized
  drift and non-smooth dispersion functions.
\newblock {\em preprint}, 2013.
\newblock arXiv:1312.7257v1 [math.PR].

\bibitem{Kry-Ro}
N.~V. Krylov and M.~R{\"o}ckner.
\newblock Strong solutions of stochastic equations with singular time dependent
  drift.
\newblock {\em Probab. Theory Related Fields}, 131(2):154--196, 2005.

\bibitem{o}
Y.~Ouknine.
\newblock Le ``{S}kew-{B}rownian motion'' et les processus qui en d\'erivent.
\newblock {\em Teor. Veroyatnost. i Primenen.}, 35(1):173--179, 1990.

\bibitem{pazy83}
A.~Pazy.
\newblock {\em Semigroups of linear operators and applications to partial
  differential equations}, volume~44 of {\em Applied Mathematical Sciences}.
\newblock Springer-Verlag, New York, 1983.

\bibitem{portenko}
N.~I. Portenko.
\newblock {\em Generalized diffusion processes}, volume~83 of {\em Translations
  of Mathematical Monographs}.
\newblock American Mathematical Society, Providence, RI, 1990.
\newblock Translated from the Russian by H. H. McFaden.

\bibitem{r1}
T.~Runst and W.~Sickel.
\newblock {\em Sobolev spaces of fractional order, {N}emytskij operators, and
  nonlinear partial differential equations}, volume~3 of {\em de Gruyter Series
  in Nonlinear Analysis and Applications}.
\newblock Walter de Gruyter \& Co., Berlin, 1996.

\bibitem{trutnau2}
F.~Russo and G.~Trutnau.
\newblock About a construction and some analysis of time inhomogeneous
  diffusions on monotonely moving domains.
\newblock {\em J. Funct. Anal.}, 221(1):37--82, 2005.

\bibitem{russo_trutnau07}
F.~Russo and G.~Trutnau.
\newblock Some parabolic {PDE}s whose drift is an irregular random noise in
  space.
\newblock {\em Ann. Probab.}, 35(6):2213--2262, 2007.

\bibitem{sv}
D.~W. Stroock and S.~R.~S. Varadhan.
\newblock {\em Multidimensional diffusion processes}, volume 233 of {\em
  Grundlehren der Mathematischen Wissenschaften [Fundamental Principles of
  Mathematical Sciences]}.
\newblock Springer-Verlag, Berlin, 1979.

\bibitem{triebel78}
H.~Triebel.
\newblock {\em Interpolation theory, function spaces, differential operators},
  volume~18 of {\em North-Holland Mathematical Library}.
\newblock North-Holland Publishing Co., Amsterdam, 1978.

\bibitem{V}
A.~Y. Veretennikov.
\newblock Strong solutions and explicit formulas for solutions of stochastic
  integral equations.
\newblock {\em Math. USSR Sb.}, 39:387--403, 1981.

\bibitem{z}
A.~K. Zvonkin.
\newblock A transformation of the phase space of a diffusion process that will
  remove the drift.
\newblock {\em Mat. Sb. (N.S.)}, 93(135):129--149, 152, 1974.

\end{thebibliography}

\end{document}